\documentclass{amsart}
\usepackage{amssymb, amsfonts, lacromay}\usepackage{mathrsfs,comment}
\usepackage[usenames,dvipsnames]{color}
\usepackage[normalem]{ulem}
\usepackage{url}
\usepackage{tikz-cd}
\usepackage[all,arc,2cell]{xy}
\UseAllTwocells
\usepackage{enumerate}

\usepackage[all,arc,2cell]{xy}
\UseAllTwocells

\usepackage{enumerate}
\usepackage{subcaption}
\usetikzlibrary{calc}
\tikzset{caption/.style={execute at end picture={\path
let \p1=($(current bounding box.east)-(current bounding box.west)$) in
(current bounding box.south) node[below,text width=\x1-4pt,align=center] 
{#1};}}}

\usepackage{todonotes}

\usepackage{hyperref}  
\hypersetup{%
  bookmarksnumbered=true,%
  bookmarks=true,%
  colorlinks=true,
  linkcolor=blue,%
  citecolor=blue,%
  filecolor=blue,%
  menucolor=blue,%
  pagecolor=blue,%
  urlcolor=blue,%
  pdfnewwindow=true,%
  pdfstartview=FitBH}   
  
\usepackage[inline]{showlabels}

\def\makeautorefname#1#2{\expandafter\def\csname#1autorefname\endcsname{#2}}
\def\equationautorefname~#1\null{(#1)\null}
\makeautorefname{footnote}{footnote}%
\makeautorefname{item}{item}%
\makeautorefname{figure}{Figure}%
\makeautorefname{table}{Table}%
\makeautorefname{part}{Part}%
\makeautorefname{appendix}{Appendix}%
\makeautorefname{chapter}{Chapter}%
\makeautorefname{section}{Section}%
\makeautorefname{subsection}{Section}%
\makeautorefname{subsubsection}{Section}%
\makeautorefname{theorem}{Theorem}%
\makeautorefname{thm}{Theorem}%
\makeautorefname{sta}{Statement}%
\makeautorefname{cor}{Corollary}%
\makeautorefname{lem}{Lemma}%
\makeautorefname{prop}{Proposition}%
\makeautorefname{pro}{Property}
\makeautorefname{conj}{Conjecture}%
\makeautorefname{defn}{Definition}%
\makeautorefname{notn}{Notation}
\makeautorefname{notns}{Notations}
\makeautorefname{rem}{Remark}%
\makeautorefname{rmk}{Remark}%
\makeautorefname{rems}{Remarks}%
\makeautorefname{quest}{Question}%
\makeautorefname{exmp}{Example}%
\makeautorefname{ax}{Axiom}%
\makeautorefname{claim}{Claim}%
\makeautorefname{assn}{Assumption}%
\makeautorefname{asses}{Assumptions}%
\makeautorefname{con}{Construction}%
\makeautorefname{prob}{Problem}%
\makeautorefname{warn}{Warning}%
\makeautorefname{obs}{Observation}%
\makeautorefname{conv}{Convention}%
\makeautorefname{cont}{Context}%

\newtheorem{thm}{Theorem}[section]

\newtheorem{cor}{Corollary}[section]
\newtheorem{prop}{Proposition}[section]
\newtheorem{lem}{Lemma}[section]

\theoremstyle{definition}
\newtheorem{defn}{Definition}[section]

\newtheorem{con}{Construction}[section]

\newtheorem{notn}{Notation}[section]

\newtheorem{rem}{Remark}[section]

\newcounter{assn}[section]
\renewcommand{\theassn}{\Alph{assn}}

\makeatletter
\let\c@cont=\c@thm
\let\c@conv=\c@thm
\let\c@obs=\c@thm
\let\c@sta=\c@thm
\let\c@cor=\c@thm
\let\c@prop=\c@thm
\let\c@lem=\c@thm
\let\c@prob=\c@thm
\let\c@con=\c@thm
\let\c@conj=\c@thm
\let\c@defn=\c@thm
\let\c@notn=\c@thm
\let\c@notns=\c@thm
\let\c@exmp=\c@thm
\let\c@ax=\c@thm
\let\c@pro=\c@thm
\let\c@ass=\c@thm
\let\c@warn=\c@thm
\let\c@rem=\c@thm
\let\c@sch=\c@thm
\let\c@equation\c@thm
\numberwithin{equation}{section}
\makeatother

\definecolor{orange}{rgb}{1,0.5,0}

\newcommand{\Xk}{X^{\otimes k}}

\newcommand{\Ek}{\sE^{\boxtimes k}}
\newcommand{\comm}{\ensuremath{\sN}}
\newcommand{\fr}{\ensuremath{\mathrm{fr}}}
\newcommand{\Mfr}{\sD_M^{\mathrm{fr}}}
\newcommand{\Vfr}{\sD_V^{\mathrm{fr}}}
\newcommand{\FM}{\underline{G\sT}(M^+, {^*}\!S^V)}

\newcommand{\Mbfr}{\bD_M^{\mathrm{fr}}}
\newcommand{\Vbfr}{\bD_V^{\mathrm{fr}}}

\newcommand{\bi}{\mathbf{i}}
\newcommand{\bj}{\mathbf{j}}
\newcommand{\bk}{\mathbf{k}}
\newcommand{\bm}{\mathbf{m}}
\newcommand{\bn}{\mathbf{n}}
\newcommand{\bp}{\mathbf{p}}
\newcommand{\bq}{\mathbf{q}}

\newcommand{\mI}{\mathcal{I}}

\title{Unital operads, monoids, monads, and bar constructions }

\author{J. P. May, Ruoqi Zhang, and  Foling Zou}

\begin{document}

\begin{abstract}
We give a description of unital operads in a symmetric monoidal category as monoids in a monoidal category of unital $\LA$-sequences.  This is a new variant of Kelly's old description of operads as monoids in the monoidal category of symmetric sequences.  The monads associated to unital operads are the ones of interest in iterated loop space theory and factorization homology, among many other applications.  Our new description of unital operads allows an illuminating comparison between the two-sided monadic bar constructions used in such applications and ``classical'' monoidal two-sided bar constructions.  It also allows a more conceptual understanding of the scanning map central to non-abelian Poincar\'e duality in factorization homology. 
\end{abstract}

\maketitle

The authors dedicate this paper to Max Kelly (1930--2007).  His insights from the very beginnings of operad theory were the inspiration for this paper.

\tableofcontents

\section*{Introduction}

The senior author defined operads in 1971 \cite{MayGeo}, focusing on operads of spaces.  However, Max Kelly was visiting Chicago at the time, and in a widely but spottily distributed 1972 preprint, only published in 2005 \cite{Kelly0}, he first wrote down the immediate observation that operads can be defined in any symmetric monoidal category $\sV$.  He also gave a conceptual description of operads as monoids in the monoidal category of symmetric sequences in $\sV$.  That description was subsequently rediscovered by several others.  It is especially convenient when studying model structures on categories of operads or of algebras over operads \cite{BMAx, BergMoerOp2, Kro, PS}.  It has been used to good effect in a number of other works, especially in the use of operads in algebra \cite{Fresse0, LV}.

Work by the third author on factorization homology led us to a rethinking of
this material.   The theoretical focus in the cited papers is on the category
$\SI$ with objects $n\geq 0$ (or just $n\geq 1$) and morphisms the symmetric
groups. A functor $\SI^{op}\rtarr \sV$ is a symmetric sequence\footnote{Some
  authors use the term collection for symmetric sequences. \\ \indent In this
  paper,  it is  important to distinguish covariant and contravariant functors
  on $\SI$, which are given respectively by left actions and right actions by
  symmetric groups.}  in $\sV$, and operads are symmetric sequences with additional structure.   However, the operads of most interest to us have a more natural underlying structure.  At the risk of belaboring the obvious, we shall give a detailed exposition of some of the basic foundational features of operad theory.

Let $\mathbf{n} =\{0,1, \cdots, n\}$, with basepoint $0$.  Call the category of such based finite sets and based injections $\LA$.  Note that the category $\SI$ with the same objects and their (based) permutations is the subcategory of isomorphisms of $\LA$, and $\LA$ is generated by the permutations and the ordered injections, or just by those ordered injections  $\si_i\colon \bf{n-1}\rtarr \bf{n}$, $1\leq i\leq n$, that skip $i$ in the target.  

Let $(\sV, \otimes, \mI)$ be a complete and cocomplete symmetric monoidal category; completeness will not be important to us in general, but in addition to cocompleteness we require that the functor $X \otimes -$ commutes with finite colimits, as holds automatically if $\sV$ is closed.  The category $\sV$ has three distinguished objects:
\begin{enumerate}[(i)]
\item $\emptyset$, the initial object, that is, the coproduct of the empty set of objects.
\item $\ast$, the terminal object, that is, the product of the empty set of objects.
\item $\mI$, the unit object for the product $\otimes$ on $\sV$.
\end{enumerate}

Let $\LA_{>0}$ be the category with objects the sets $\{1,\cdots,n\}$ for $n\geq 1$ and morphisms the injections.
A covariant functor $\LA_{>0}\rtarr \sV$ can be identified with a covariant functor $\LA \rtarr \sV$ with zeroth value $\emptyset$.  
A contravariant functor $\LA_{>0} \rtarr \sV$ can be identified with a contravariant functor $\LA\rtarr \sV$ with zeroth value $\ast$.   

\begin{rem}  The notation $\LA$ was introduced in 1978 \cite{CMT}, where a contravariant functor defined on $\LA$ is called a 
coefficient system.  In  a 1997 paper by Berger \cite{Berger}, the notation $\LA$ is used for our $\LA_{>0}$ and a contravariant functor 
defined on $\LA_{>0}$ is called a preoperad; those conventions are also followed in the 2002 book of Markl, Shnyder, and Stasheff \cite{MSS}.  In much recent work, by Church, Farb, Ellenberg, and others, $\LA_{>0}$ is denoted $FI$  \cite{CEF, CEFN}.
 \end{rem}

Of course, $\ast= \mI$ when $\sV$ is cartesian monoidal.  However, in many algebraic examples, such as modules over a commutative ring, we instead have $\emptyset=\ast$, and that object is then a zero object and is denoted $0$.  It is usual to define a based object $X$ in $\sV$ to have a map $\ast\rtarr \sV$, but that is clearly not a useful concept in such algebraic examples.   The following definition gives the appropriate version of based objects in our general context.

\begin{defn}\label{unital}  Define a based object $(X,\et)$ in $\sV$ to be an object $X$ of $\sV$ together with a map $\et\colon \mI\rtarr X$, which we call a base map; we abbreviate notation to $X$ when $\et$ is understood.   Let $\sV_{\mI}$ denote the category of based objects, that is, objects under $\mI$.
\end{defn}

Choosing one of our three choices for $\sC(0)$ defines three special kinds of covariant or contravariant functor $\sC$ from $\LA$ to $\sV$.  The literature is marred by conflicting names for different choices.   When $\sV$ is cartesian monoidal, it is usual to say that $\sC$ is {\em reduced} if $\sC(0)=\ast$.  Following \cite{MayOp1}, we say that $\sC$ is {\em unital} if $\sC(0) = \mI$.\footnote{Some authors define unital operads to be those that have a unit map $\mI\rtarr \sC(1)$.  For us, {\em all} operads are required to be unital in that sense,  as in the original definition. Such a unit map is essential to the connection between operads and monads that is the source of the name ``operad".}   Thus reduced and unital are the same when $\sV$ is cartesian monoidal, but they are not the same in general and we shall not use the term reduced in this paper.  

\begin{defn}  We say that an operad $\sC$ in $\sV$ is {\em based} if $\sC(0)$ is based; we say that $\sC$ is unital if $\sC(0)= \mI$.
\end{defn}

\begin{rem}  Operads of spaces that are based but not necessarily unital play a prominent role in \cite{BBPTY}. Other examples appear in \cite[\S7]{GM3}.   Endomorphism operads of based objects (such as based spaces) are always based but almost never unital.  We shall find a context in which endomorphism operads in enriching categories can be unital in \autoref{VWcon}, and an example in \autoref{fact}  will play a key role in factorization homology.  
\end{rem}

Operads with $\sC(0)=\emptyset$ appear naturally when dealing with algebraic structures that lack unit objects, such as Lie algebras.  However, our focus in this paper is on algebraic structures that do have unit objects, and there we take $\sC(0) = \mI$.   We focus on unital operads here since we want a unique choice of base object.  However, most of our results work more generally for based operads.  

\begin{rem} Terminology in the literature is quite inconsistent. In some sources, such as the book \cite{MSS} and Ching's paper \cite{Ching}, operads are defined to start with $\sC(1)$, with no term $\sC(0)$.  In \cite{Ching}, operads without zero terms are said to be unital if $\sC(1) = \mI$.  That gives a quite different context from the one we study here.  In the language of Fresse \cite[1.1.19]{Fresse}, followed more recently in \cite{GKRW}, those non-unital operads with $\sC(0) = \emptyset$ or, equivalently, with no zero term are called non-unitary operads, and our unital operads are called unitary operads; Fresse defines connected operads to be what Ching calls unital operads.   Berger and Moerdijk \cite{BMAx}, among others, use the term reduced for what we are calling unital. \end{rem}

The purpose of defining operads is to define algebras over them.   Our initial focus on the choice of $\sC(0)$ is in part motivated by the following observation.

\begin{lem}  For an operad $\sC$, the object $\sC(0)$ of $\sV$ is a $\sC$-algebra, and it is an initial object in the category of $\sC$-algebras in $\sV$.
\end{lem}
\begin{proof}  The structure maps $\ga\colon \sC(j)\otimes \sC(0)^{\otimes j} \rtarr \sC(0)$ give the action of $\sC$ on $\sC(0)$.  Using the standard convention that $X^0 = \mI$, we see that the component $\tha_0\colon  \sC(0) \iso \sC(0)\otimes X^0\rtarr X$ of an action $\tha$ of $\sC$ on $X$ gives the unique map  of $\sC$-algebras   $\sC(0)\rtarr X$.   Details are immediate from the definitions of operads and their algebras \cite{MayOp1}. 
\end{proof}

\begin{rem}   For an operad $\sC$ we can obtain a non-unitary operad $\sC^{-}$ by replacing $\sC(0)$ with $\sC^{-}(0) = \emptyset$.  Its algebras are the non-unital $\sC$-algebras.   For a non-unitary operad $\sC$, we can attempt to obtain a unital operad $\sC^{+}$ by replacing $\sC(0) = \emptyset$ with $\sC^{+}(0) = \mI$, but that does not always give an operad.  This idea is used in \cite{GKRW}.
\end{rem}

We give names for some of the categories that will be of interest to us.

\begin{notn}\label{notn1}  Let $\LA[\sV]$ and $\LA^{op}[\sV]$ denote the respective categories of covariant and contravariant functors 
$\LA\rtarr \sV$ and define  $\SI[\sV]$ and $\SI^{op}[\sV]$ analogously.   Let $\LA[\sV_{\mI}]$ and $\LA^{op}_{\mI}[\sV]$ denote the respective categories of unital covariant and unital contravariant functors $\LA\rtarr \sV$; in both cases, the unital condition requires a 
functor $\sD$ to send $\bf 0$ to $\mI$.   In the covariant case, but not in the contravariant case, use of the unique morphism 
$\mathbf{0}\rtarr \bn$ in $\LA$ ensures that the functor $\sD$ takes values in $\sV_{\mI}$, hence the difference of notation depending on the variance.
\end{notn}

\begin{notn}\label{notn2} We call \emph{contravariant} functors defined on $\SI$ symmetric sequences or $\SI$-sequences.  We introduce the name $\LA$-sequence for a \emph{contravariant} functor $\sD$ defined on $\LA$ together with a base map 
 $\et\colon \mI \rtarr \sD (\mathbf 0)$, and we write 
$\LA^{op}[\sV]_{\mI_0}$ for the category of $\LA$-sequences, where the morphisms are natural transformations that restrict to morphisms under $\mI$ at level $\mathbf{0}$. The notation $\mI_0$ indicates the $\LA$-sequence that is $\mI$ at level $\mathbf{0}$ and $\emptyset$ at levels $\bn$ for $n>0$.  A map  $\mI\rtarr \sD (\mathbf{0})$ is the same as a map of $\LA$-sequences $\mI_0\rtarr \sD$, and $\mI_0$ plays the  same role in $\LA^{op}[\sV]_{\mI_0}$ that $\mI$ plays in $\sV_{\mI}$.   We say that a $\LA$-sequence $\sD$ is unital if $\et\colon \mI\rtarr  \sD (\mathbf{0})$ is the identity map. We have an inclusion of categories  $\LA^{op}_{\mI}[\sV]\subset \LA^{op}[\sV]_{\mI_{0}}$. 
\end{notn}

We recall the following observation in \autoref{Cstar}.

\begin{lem}  A based operad $\sC$ in $\sV$ has an underlying $\LA$-sequence, hence a unital operad has an underlying unital $\LA$-sequence.
\end{lem}

\begin{rem} The forgetful functor from operads to symmetric sequences and the forgetful functor from unital operads to unital $\LA$-sequences have left adjoint free operad functors.  These are studied in \cite[\S 1.2, \S  2.3, and  \S A.3 ]{Fresse}. 
\end{rem}

A theme of this paper is an analogy between the relationship that holds between unbased and based spaces in topology and the relationship that holds between $\SI$-sequences and $\LA$-sequences, especially unital ones, in our general operadic context.  For reasons analogous to those for focusing  on based rather than unbased spaces in much of algebraic topology, it is natural to focus on $\LA$-sequences rather than $\SI$-sequences in much of operad theory.  

In \autoref{prel}, we briefly recall some categorical definitions, focusing on tensor products of functors, which are examples of left Kan extensions.
In \autoref{mon1}, we give preliminaries about the structure of $\LA^{op} [\sV]_{\mI_{0}}$ and its subcategory $\LA^{op}_{\mI}[\sV]$ as symmetric monoidal categories under Day convolution,
which is a standard example of a left Kan extension. In \autoref{comb}, we prove the surprising categorical fact that Day convolution of $\LA$-sequences agrees with Day convolution of their underlying $\SI$-sequences; see \autoref{boxbox}.  In \autoref{mon}, we give $\LA^{op}[\sV]_{\mI_0}$ and 
its subcategory $\LA^{op}_{\mI}[\sV]$ different monoidal structures, not symmetric,  just as Kelly gave $\SI^{op}[\sV]$ a monoidal structure. Kelly's theorem can be stated as follows.  

\begin{thm}\label{Kelly0}[Kelly]  The category of operads in $\sV$ is isomorphic to the category of monoids in the monoidal category $\SI^{op}[\sV]$ of $\SI$-sequences.
\end{thm}

The relevant product in $\SI^{op}[\sV]$ is often called the composition product, but we shall call it the Kelly product, in honor of its inventor.
Using \autoref{boxbox}, we shall prove the following analog in \autoref{oper}.  

\begin{thm}\label{FRP} The category of based operads in $\sV$ is isomorphic to the category of monoids in the monoidal category 
$\LA^{op}[\sV]_{\mI_0}$ of $\LA$-sequences.  Therefore the category of unital operads in $\sV$ is isomorphic to the category of monoids in the monoidal category $\LA^{op}_{\mI}[\sV]$ of unital $\LA$-sequences.
\end{thm}

\begin{rem}\label{senior}  The senior author glibly asserted the second statement, without proof, in 1997 \cite[p. 5]{MayOp1}.   The proof is not as obvious as he thought then.  
\end{rem}

Recall that the original point of the name ``operad'' was that an operad $\sC$ in a suitable ground category  $\sW$ has an associated monad $C$ in $\sW$ such that the categories of $\sC$-algebras and $C$-algebras are isomorphic \cite{MayGeo, MayOp2}. We relate the Kelly product to the monadic point of view in \autoref{mmcomp}.  The point of the cited isomorphism of categories is that it allows exploitation of the two sided {\em monadic} bar construction to study algebras over operads.   Topologically, that is the context for the senior author's 1970's  approach to iterated loop space theory, but many other applications have arisen since, such as the geometric approach to factorization homology, which we shall discuss briefly  in \autoref{fact}.  

The choice of $\sW$ matters, and there can be more than one choice leading to monads in different ground categories with isomorphic categories of algebras \cite[Section 4]{Rant1}.  This typically happens when we start with a unital operad 
$\sC$  and compare the associated monads  in $\sV$ and in $\sV_{\mI}$.   One way of thinking about the difference is that in one we are implicitly thinking of an operad $\sC$ as a monoid in $\SI^{op}[\sV]$ and in the other we are implicitly thinking about $\sC$ as a monoid in 
$\LA^{op}_{\mI}[\sV]$. 

We now have an embarrassment of riches, explored in \autoref{bar}.  The two-sided {\em monoidal} bar construction first appeared implicitly in 1940's homological work of Eilenberg and MacLane.  In any monoidal category, $(\sW,\otimes, \mI_{\sW})$ say,  we can start with a monoid   $M$ in $\sW$ with a right action on an object  $Y$ and a left action on an object $X$.  We then have a simplicial object 
\[   B_*(Y,M,X) \]
in $\sW$  whose $q$-simplex object is 
\[  Y \otimes M^{\otimes q} \otimes  X. \]
On $q$-simplices, $d_0$ is given by the action of $M$ on $Y$,  $d_i$ for $1\leq i \leq q-1$ is given by the product  on $M$ in the $i$th position, $d_q$ is given by the action of $M$ on $X$, and $s_i$ for $0\leq i\leq q$ is given by inserting the unit $\mI_{\sW} \rtarr M$ in the $i$th position.  In many interesting examples, there is a realization functor from simplicial objects in $\sW$ to $\sW$, such as geometric realization in topological contexts and totalization of simplicial chain complexes in homological contexts.   For an important example in which such a direct realization functor is of little or no interest,  take $\sW$ to be the monoidal category of graphs (in the categorical sense, with source and target) with a fixed set of objects; a monoid in $\sW$ is then a category, and right and left actions are contravariant and covariant functors; see \cite[\S9]{MayClass}.  We shall ignore realization until \autoref{fact}.

There is a considerable algebraic literature that explores bar constructions in
the context of operads considered as monoids in monoidal categories of symmetric
sequences.  Fresse's book \cite{Fresse0} and Loday and Vallette's book \cite{LV}
are standard references.  Ching \cite{Ching} and others have studied related
variants.   However, as far as we know, our variant of working in the
  monoidal category of $\LA$-sequences is new. 

We are especially interested in understanding the relationship between the operadic specialization of {\em monoidal} simplicial bar constructions and the {\em monadic} simplicial bar constructions $B_*(F, C, X)$ that were central to the original theory of operads.   Here $X$ is a $C$-algebra and $F$ is a $C$-functor landing in some category, not necessarily the ground category of the monad $C$.   The two bar constructions look quite different, and the comparision is not at all obvious at first sight.  A precursor to this paper by the second author  \cite{Rachel} explores this in the context of symmetric sequences.  Although the monoidal bar construction admits greater generality than we shall pursue, we shall see that in practice it is often a specialization of the monadic bar construction. That specialization seems to be insufficiently general for the relevant topological applications; see \autoref{intex}.   Pursuing comparisons in algebraic contexts deserves further study.

It is classical that, when $\sV$ is closed, an algebra $X$ over an operad $\sC$ is the same as a map of operads from $\sC$ to the endomorphism operad $\End(X)$ of $X$.   Heading towards factorization homology, we show in \autoref{VWcon} how to generalize this statement to operads in a category $\sW$ enriched in $\sV$ and we investigate the analogous statement for unital operads and their algebras.   In \autoref{fact}, we specialize these ideas  to give a conceptual framework for geometric factorization homology in the equivariant context developed by the third author and for her proof of nonabelian Poincar\'e duality.  The construction of the scanning map that gives that duality isomorphism seems to us to be of particular interest. It is worked out conceptually in \autoref{SCAN}.

In the brief \autoref{remarks}, we remark on where the ideas here, in particular the focus on $\LA$-sequences rather than just on $\SI$-sequences, are relevant to other work  in the literature.  The topics are categories of modules over the commutativity operad, model categories of operads, generalizations of the Kelly product, and the role of $\LA$ in the theory of dendroidal sets.

While our focus is on operads in $\sV$ viewed as monoids in non-symmetric monoidal categories, our details have led us to a better understanding of an alternative perspective that views operads as symmetric monoidal categories enriched in $\sV$.   This point of view is folklore that is articulated in concrete terms in \cite[Proposition 3.1]{BBPTY2}.  It is also at least implicit in the specialization to operads of the symmetric monoidal envelopes of multicategories (e.g.  \cite[p. 291]{CGw}), which is itself very closely related to the earlier construction of categories of operators associated to operads \cite{MT}.  We discuss these ideas in the Appendix, \autoref{SMC}, describing symmetric monoidal envelopes in \autoref{SME} and their relationship to categories of operators in \autoref{COO}.

We emphasize that  $\sV$ is general throughout.  Our immediate interest is in spaces and $G$-spaces for a group $G$, but we envision algebraic and homological applications.

\section*{Acknowledgement}
The authors thank the referee for many useful suggestions to improve the paper.

\section{Categorical preliminaries}\label{prel}

We shall work throughout  in a bicomplete symmetric monoidal category $(\sV,\otimes, \mI)$ such that $\otimes$ commutes with finite colimits.  Since familiarity with certain standard categorical constructions is essential, we give a brief review of what we shall need.   Mac\,Lane advertised that ``all concepts are Kan extensions''.  We focus on those left Kan extensions which are tensor products of functors, and those are examples of coends.  Thus for us all concepts are coends or, more particularly, tensor products of functors.  It will be helpful to us to be maximally explicit and minimally abstract since the details of the constructions will matter.  See  \cite{Mac, Riehl} for more details of the categorical context.

\begin{defn}\label{coend}  Let  $\GA$ be a category and let  $H\colon \GA^{op}\times \GA \rtarr \sV$
  be a functor.   The {\em coend}  $\int^{c\in \GA} H(c,c)$ is the coequalizer in $\sV$ displayed in the diagram
\[  \xymatrix{
\coprod_{f\colon a \rtarr b}  H(b,a) \ar@<1ex>[r]^-{\al}  \ar@<-1ex>[r]_-{\be} &  \coprod_{c} H(c,c) \ar[r] & \int^{c\in \GA} H(c,c)\\} \]
where $a,b,c$ range over the objects of $\GA$, $f$ ranges over the morphisms of $\GA$, and where, on the component indexed by $f$, 
\[  \al = H(f,\id)\colon H(b,a) \rtarr H(a,a) \ \ \text{and} \ \  \be = H(\id,f)\colon H(b,a) \rtarr H(b,b).   \]
We abbreviate notation to $\int^{\GA}H$ whenever convenient.  
\end{defn}

Of course, the universal property can be read off from the definition, which depends only on the assumption that $\sV$ is cocomplete.
Using the further assumption that $\sV$ is monoidal (ignoring symmetry), we arrive at the examples of interest.

\begin{defn}\label{coendtoo}   Define the {\em tensor product}  $G\otimes_{\GA} F$ of functors $G\colon \GA^{op} \rtarr \sV$ and $F\colon \GA\rtarr \sV$ by
\[  G\otimes_{\GA} F =  \int^{c\in \GA} G(c)\otimes F(c); \]
Thus $H$ here is $G\otimes F$, namely the composite  $\otimes \com (G\times F)\colon  \GA^{op}\times \GA\rtarr \sV\times \sV\rtarr \sV$.
\end{defn}

Again, the universal property can be read off from the definition.   To be explicit about  general left  Kan extensions, we use the following observation to make sense of the tensor of objects of $\sV$ with sets.

\begin{rem}\label{set}  We embed the category of sets in $\sV$ by the functor that sends a set $S$ to the coproduct, denoted $\mI[S]$, of copies of $\mI$ indexed on the elements of $S$.  However, we shall use this implicitly and just regard sets as objects of $\sV$.  Then $\sV$ is tensored over sets via $S\otimes X$ for  $X$ in $\sV$. In particular, we view the morphism sets  of (ordinary locally small) categories as objects of $\sV$ without change of notation.   We shall instead write $X\otimes S$ when that is more convenient.
\end{rem}

\begin{defn}  For (covariant) functors  $F\colon \GA \rtarr \sV$ and $\om\colon \GA \rtarr \PS$, define the {\em left Kan extension} of $F$ over $\om$ by
\[  (\mathbf{Lan}_{\om} F)(d) = \int^{c\in \GA}  \PS(\om(c),d)\otimes F(c) \]
on objects $d\in \PS$.  Equivalently,
\[  (\mathbf{Lan}_{\om} F)(d) =  \PS(\om(-),d)\otimes_{\GA}  F. \]
Then the diagram
\[ \xymatrix{
\GA \ar[r]^{F} \ar[d]_{\om} & \sV \\
\PS\ar@{-->}[ur]_{\mathbf{Lan}_{\om} F} \\} \] 
commutes up to the natural transformation $\et_F\colon F \Rightarrow  \mathbf{Lan}_{\om} F\com \om$ that identifies $F(c)$ with 
$\id_{\om(c)}\otimes F(c)$.  The universal property can be expressed as an adjunction between functor categories
\[  \sV^{\PS}(\mathbf{Lan}_{\om} F, G) \iso \sV^{\GA}(F,G\com \om), \]
where $G$ is a (covariant) functor  $\PS \rtarr \sV$; the $\et_F$ are the components of the unit $\et$ of the adjunction.
\end{defn}

\begin{rem}  Observe that covariant functors  $\GA^{op}\rtarr \PS^{op}$ are the same as covariant functors $\GA\rtarr \PS$.  We shall often start with contravariant functors $F\colon \GA^{op}\rtarr \sV$ and covariant functors  $\GA^{op}\rtarr \PS^{op}$. It is then convenient to reverse the order of variables, writing
\[   (\mathbf{Lan}_{\om} F)(d) = \int^{c\in \GA}  F(c)\otimes  \PS(d,\om(c)) = F\otimes_{\GA} \PS(d,\om(-)).  \]
\end{rem}

Using that notation,  the co-Yoneda lemma reads as follows in terms of coends. 
\begin{lem}\label{coYon}
For a functor $G\colon \GA^{op} \rtarr \sV$,
$$G \iso \int^{c\in \GA} G(c) \otimes \GA(-, c)$$
\end{lem}

\section{The Day convolution symmetric monoidal structures}\label{mon1}

We begin with the following observation.

\begin{lem}\label{Xstar}  If $X$ is a based object in $\sV$, then the tensor powers  $\Xk$, where $X^{\otimes 0} = \mI$,  give the values on objects of a covariant functor $X^{\otimes *}\colon \LA\rtarr \sV_{\mI}$.
\end{lem} 
Proof. We let $\SI_k$ act from the left on $X^{\otimes k}$ by permutation of variables. We define 
$\si_i \colon X^{\otimes k-1} \rtarr X^{\otimes k}$  to be
$$\id^{i-1}\otimes \et \otimes \id^{k-i} \colon X^{\otimes k-1} \iso X^{\otimes i-1} \otimes \mI \otimes X^{\otimes k-i} \rtarr X^{\otimes k}. $$
As in \autoref{notn1}, the resulting functor takes values in $\sV_I$. 

\begin{rem}\label{pedant}  The wedge $\vee$ of finite based sets gives $\LA$ and also $\LA^{op}$ a structure of permutative category. Pedantically, if $p+q = n$, write $\bf{p+q}$ for  $\bn$ regarded as the one-point union of the sets $\{0,1,\cdots,p\} $ and $\{0, p+1, \cdots, p+q\}$. Then the functor $\wed$ is specified by identifying the  wedge $\bp\wed \bq$ with the object $\bf{p+q}$ of $\LA$. The commutativity isomorphism of $\wed$ is given by the evident isomorphisms $\mathbf{p+q}\iso \mathbf{q+p}$.  Note that $X^{\otimes *}$ is a symmetric monoidal functor $\LA\rtarr \sV$.   Passage from $X$ to $X^{\otimes *}$ gives a covariant functor $\sV_{\mI}\rtarr \LA[\sV_{\mI}]$.  
\end{rem}

Functor categories with domain a small symmetric monoidal category and target a symmetric monoidal category are themselves symmetric monoidal under Day convolution. Our focus is on the following example.  Recall that $\mI_0(\mathbf{0}) = \mI$ and $\mI_0(\bn) = \emptyset$ for $n>0$.

\begin{prop}  Under Day convolution, denoted $\boxtimes$, the category $\LA^{op}[\sV]$ is symmetric monoidal with unit $\mI_0$.
\end{prop}

We recall the definition of $\boxtimes$. For $\sD$ and $\sE$ in  $\LA^{op}[\sV]$, we have their external tensor product
\[ \sD \overline{\otimes} \sE\colon \LA^{op} \times \LA^{op} \rtarr \sV \]
given by  $(\sD \overline{\otimes} \sE)(\bj,\bk) = \sD(\bj)\otimes \sE(\bk)$.   The Day convolution $\sD\boxtimes \sE$ is the left Kan extension
displayed in the diagram
\begin{equation}\label{Kan1}
\xymatrix{  \LA^{op}\times \LA^{op} \ar[rr]^-{\sD\overline{\otimes} \sE} \ar[d]_{\bigvee}& & \sV.\\
\LA^{op} \ar@{-->}[urr]_{\sD\boxtimes \sE} & & \\}
\end{equation}
That $\mI_0$ is the unit is dictated by general theory and can be checked by a quick inspection of definitions, starting from the observation that
\[  (\mI_0\overline{\otimes} \sE)(\bj,\bk) = \left \{\begin{array}{ll} 
\mI\otimes \sE(\bk)\iso \sE(\bk) & \mbox{if $j=0$} \\
 \emptyset & \mbox{if $j>0$}  
 \end{array} \right.  \]
 for any $\sE\in \LA^{op}[\sV]$.

We will be interested in both the universal property 
\begin{equation}\label{otimes}
\sV^{\LA^{op}}(\sD\boxtimes \sE, \sF)\iso \sV^{\LA^{op}\times \LA^{op}}(\sD\overline{\otimes} \sE,\sF\com \vee) 
\end{equation}
and the explicit construction of $\boxtimes$. 

\begin{lem}\label{subbox} The Day convolution on $\LA^{op}[\sV]$  gives $\LA^{op}[\sV]_{\mI_0}$ and its subcategory 
$\LA^{op}_{\mI}[\sV]$ symmetric monoidal structures with unit  $\mI_0$.  
\end{lem}
\begin{proof} It follows from either the universal property or the explicit construction of $\boxtimes$ that if $\sD$ and $\sE$ are based, then their base maps induce a base map $\mI_0\rtarr \sD\boxtimes \sE$.   If $\sD$ and $\sE$ are unital, then inspection of the explicit construction shows that 
$\sD\boxtimes\sE$ is also unital, meaning that $(\sD\boxtimes\sE)(\mathbf{0})$ can also be identified with $\mI$. 
\end{proof}

We introduce the following complement to Notations \ref{notn1} and \ref{notn2}. 

\begin{notn}\label{notn3}  Let  $i_0\colon \sV_{\mI} \rtarr \LA^{op}[\sV]_{\mI_0}$ be the inclusion of categories that 
sends $(X,\et)$ in $\sV_{\mI}$ to the functor $i_0X \colon \LA^{op} \rtarr \sV$ specified by 
$$(i_0 X)(\mathbf{0}) = X \ \ \text{and} \ \  (i_0X)(\bk) = \emptyset  \ \ \text{for} \ \ k>0; $$
the base map $\et\colon \mI \rtarr X$ determines the required base map $\et\colon \mI_0 \rtarr i_0X$. 
Let $p_0\colon  \LA^{op}[\sV]_{\mI_0}\rtarr \sV_{\mI}$ be the $0$th object functor.  Note that
$(i_0,p_0)$ is an adjoint pair,
\[    \LA^{op}[\sV]_{\mI_0}(i_0 X,\sD) \iso \sV_{\mI}(X,p_0\sD). \]  
\end{notn}

 \begin{lem}\label{strongp}  The functors $i_0$ and $p_0$ are  strong symmetric monoidal.
 \end{lem}
 \begin{proof}  For $i_0$, this  follows from the universal property and the identification
 \[  (i_0 X\overline{\otimes} i_0Y)(\bj,\bk) = \left \{\begin{array}{ll} 
X\otimes Y & \mbox{if $j=0$ and $k=0$} \\
 \emptyset & \mbox{if $j>0$ or $k>0$.}  
 \end{array} \right.  \]
For $p_0$,  this is clear from the concrete description of $\boxtimes$ in \autoref{boxbox} below.
\end{proof}

We view $ \LA^{op}[\sV]_{\mI_0}$ as a convenient common home for $\LA^{op}_{\mI}[\sV]$ and $\sV_{\mI}$. 
Observe that we can perfectly well apply \autoref{Xstar} with $\sV$ replaced by $\LA^{op}[\sV]$.  It then reads as follows.

\begin{lem}\label{Estar}  If $\sE$ is in $\LA^{op}[\sV]_{\mI_0}$, then the Day tensor powers  $\Ek$ give the values on objects of a covariant functor 
$\sE^{\boxtimes *}\colon \LA\rtarr \LA^{op}[\sV]_{\mI_0}$.  Here we interpret $\sE^{\boxtimes 0}$ to be $\mI_0$.  For $k>0$,  $\Ek$ is the left Kan extension
\begin{equation}\label{Kan2}
\xymatrix{  (\LA^{op})^{k} \ar[rr]^-{\sE\overline{\otimes}\cdots \overline{\otimes}\sE} \ar[d]_{\bigvee}& & \sV.\\
\LA^{op} \ar@{-->}[urr]_{\Ek} & & \\}
\end{equation}
\end{lem}  

\section{Combinatorics of Day convolution}\label{comb}

Our hypotheses on $\sV$ ensure that the functor $\otimes$ takes coproducts to coproducts in either variable, as holds automatically when $\sV$ is closed. 

Let $\sD$ and $\sE$ be objects of $\LA^{op}[\sV]_{\mI_0}$.  The Kan extensions \autoref{Kan1} and  \autoref{Kan2} are  coends
\begin{equation}\label{coend1}    
(\sD\boxtimes \sE)(\bn)  = \int^{\bj_1,\bj_2}\big( \sD(\bj_1) \otimes \sE(\bj_2) \big) \otimes  
\Lambda(\bn, \mathbf{j_1 +  j_2}),
\end{equation}
\begin{equation}\label{coend2}    
\Ek(\bn)  = \int^{\bj_1, \ldots, \bj_k}\big( \sE(\bj_1) \otimes \ldots \otimes \sE(\bj_k) \big) \otimes  \Lambda(\bn, \mathbf{j_1 + \ldots + j_k}). 
\end{equation}

\begin{rem}\label{combtoo}
The contravariant functoriality of  $\sE^{\boxtimes k}(\bn)$ on $n$ is given by precomposition in 
$\LA(-,\bj_1+ \cdots + \bj_k)$.  As $k$ varies, the covariant functoriality on $k$ is best understood by  treating permutations and proper injections separately.  Fix $(j_1, \cdots, j_k)$ with $n= j_1+\cdots + j_k$.  For  $\si\in \SI_k$, $\si$ acts diagonally as the tensor of 
$$\si\colon  \sE(\bj_1) \otimes \ldots \otimes \sE(\bj_k) \rtarr  \sE(\bj_{\si^{-1}(1)}) \otimes \ldots \otimes \sE(\bj_{\si^{-1}(k)})$$ 
with the map $\si\colon \SI_n\rtarr \SI_n$, given by postcomposition with the block permutation  of $n$  letters obtained via the lexicographic identification of $\bn$ with the two $k$-fold wedge products in sight:
$$ \si(j_1,\dots, j_k)\colon \bn\iso \bj_1 \vee \cdots \vee \bj_k \rtarr \bj_{\si^{-1}(1)} \vee \cdots \vee \bj_{\si^{-1}(k)}\iso \bn. $$
Similarly, for an injection $\si_i\colon \bk \rtarr \mathbf{k+1}$, $\si_i$ acts diagonally as the tensor of
$$ \si_i \colon   \sE(\bj_1) \otimes \ldots \otimes \sE(\bj_k) \rtarr  
\sE(\bj_1) \otimes \cdots \otimes \sE(\mathbf{0})  \otimes \cdots \otimes \sE(\bj_k) $$
obtained by inserting the unit $\mI \rtarr \sE(\mathbf{0})$ in the $i$th position and the identity map 
$$ \bj_1 \vee \cdots \vee\bj_k  \rtarr   \bj_1 \vee \cdots \vee \mathbf{0} \vee \cdots \vee \bj_k. $$
That is, we are changing from $(j_1, \cdots, j_k)$ to $(j_1, \cdots, 0, \cdots, j_k)$ by inserting $0$  in the $i$th slot.  Observe that for each fixed $n$,  the $\Ek(\bn)$ as $k$ varies do indeed give a covariant functor $\LA\rtarr \sV$. 
\end{rem}

We have analogous definitions  with $\LA$ replaced by $\SI$.  We write $\boxtimes_{\SI}$ and $\boxtimes_{\LA}$ for the respective Day tensor products in this section.  By the universal property of $\boxtimes_{\SI}$, we have a natural map of $\SI$-sequences
$\io\colon  \sD\boxtimes_{\SI} \sE \rtarr  \sD\boxtimes_{\LA} \sE$.

The following combinatorial result looks quite surprising at first sight, but it is proven by using a specialization of a standard categorical result on the comparison of colimits \cite[Lemma 8.3.4]{Riehl}.    Our original direct proof is more illuminating, but is lengthier and is  therefore just sketched. The impatient reader is urged to accept the result and skip the details of proof.
\begin{thm}\label{boxbox} For  $\sD$ and $\sE$ in  $\SI^{op}[\sV]$, 
\[  (\sD\boxtimes_{\SI} \sE)(\bn) \iso \coprod_{j+k = n} \sD(\bj) \otimes \sE(\bk) \otimes_{\SI_j\times \SI_k} \SI_n.  \]
For $\sD$ and $\sE$ in $\LA^{op}[\sV]$, the natural map
\[  \io \colon  \sD\boxtimes_{\SI} \sE \rtarr  \sD\boxtimes_{\LA} \sE \] 
is an isomorphism of $\SI$-sequences. 
\end{thm}
\begin{proof}
Recall \autoref{pedant} and remember that $\SI$ and $\LA$ have the same objects,  that $\SI_n = \SI(\bn,\bn) = \LA(\bn,\bn)$, and that $\LA(\bn,\bp)$ is empty unless $n\leq p$.

For the first statement, \cite[Remark 1.2.6]{Riehl}  shows that $\sD \boxtimes_{\SI} \sE$ can be computed pointwise as
\begin{equation}\label{boxtimesSI2} 
(\sD \boxtimes_{\SI} \sE)(\bn) =  \xymatrix@1{\mathrm{colim}\big{(} \vee_{\SI}/{\bn} \ar[r]^-{U} & \SI^{op} \times \SI^{op} \ar[r]^-{\sD\overline{\otimes}\sE} & \sV \big{)}.\\}
\end{equation}
Here $\vee_{\SI}$ denotes the functor $\vee\colon \SI^{op} \times \SI^{op} \rtarr \SI^{op}$,  $\vee_{\SI}/{\bn}$ denotes its slice category of objects over $\bn$, and
$U$ is the evident forgetful functor.  Writing in terms of $\SI$ rather than $\SI^{op}$, 
recall that the objects of  $\vee_{\SI}/{\bn}$ are triples $\{\bj,\bk;\ta\}$ where $j+k = n$ and
$\ta\colon \bn\rtarr \mathbf{j+k} = \bn$ is in $\SI_n$.
The morphisms of $\vee_{\SI}/{\bn}$  are of the form
$$(\mu,\nu) \colon \{\bj,\bk;\ta\} \rtarr \{\bj,\bk;\ta'\}$$
where $(\mu,\nu)\in \SI_j\times \SI_k$ and $(\mu\wed \nu)\com \ta = {\ta}'$.  More explicitly, we have the isomorphism
\[ (\sD\boxtimes_{\SI} \sE)(\bn) = \coprod_{j+k = n} \sD(\bj) \otimes \sE(\bk) \otimes_{\SI_j\times \SI_k} \SI_n. \]
Here, with $j+k=n$,  $\SI_{j}\times \SI_{k}$ is embedded as a subgroup of $\SI_{n}$ via block sum of permutations.  Letting $\SI_{j,k}\subset \SI_{j+k}$ be the subset of $(j,k)$-shuffles, we observe that these give canonical orbit representatives, so that, as objects of $\sV$,
\begin{equation}\label{NoOrbits}
  (\sD\boxtimes_{\SI} \sE)(\bn) \iso \coprod_{j+k = n} \sD(\bj) \otimes \sE(\bk) \otimes \SI_{j,k}.
\end{equation}

For the second statement, \cite[Remark 1.2.6]{Riehl}  shows that $\sD \boxtimes_{\Lambda} \sE$ can be computed pointwise as
\begin{equation}\label{boxtimesLA2} 
(\sD \boxtimes_{\LA} \sE)(\bn) =  \xymatrix@1{\mathrm{colim}\big{(} \vee_{\LA}/{\bn} \ar[r]^-{U} & \LA^{op} \times \LA^{op} \ar[r]^-{\sD\overline{\otimes}\sE} & \sV \big{)}.\\}
\end{equation}
Writing in terms of $\LA$ rather than $\LA^{op}$,  the slice category $\vee_{\LA}/\bn$ has objects $\{\bp,\bq,\si\}$ where $p+q\geq n$ and $\si$ is an injection $\bn\rtarr  \mathbf{p+q}$.  Its
morphisms are of the form
$$(\mu,\nu) \colon \{\bp,\bq;\si\} \rtarr \{\bp',\bq';\si'\}$$
where $\mu\colon \bp\rtarr \bp'$ and $\nu\colon \bq\rtarr \bq'$ are injections such that  $(\mu\wed \nu)\com \si  = \si'$.

Consider the following diagram.
\begin{equation*}
  \xymatrix{
    \SI^{op} \times \SI^{op}  \ar[r]^{i^{op} \times i^{op}} \ar[dd]_{\bigvee_{\SI}} &
    \LA^{op} \times \LA^{op}  \ar[rr]^-{\sD\overline{\otimes}\sE}     \ar[d]_{\bigvee_{\LA}} & & \sV.\\
 & \LA^{op} \ar@{-->}[urr]_-{\sD\boxtimes_{\LA}\sE} &  \dltwocell<\omit>{<0>   \io}& \\
     \SI^{op} \ar[ur]^{i^{op}} \ar@{-->}@/_4pc/[uurrr]_-{\sD\boxtimes_{\SI}\sE}  & & \\}
\end{equation*}
The map $\io$  is induced pointwise by the map of slice categories
$$ i^{op}/\bn \colon  \wed_{\SI}/\bn \rtarr \wed_{\LA}/\bn$$
induced by the trapezoid on the left.

It suffices to show that each $i^{op}/\bn$ is final, in the sense of \cite[Definition 8.3.2]{Riehl}. By \cite[Lemma 8.3.4]{Riehl}, this is true if and only if the
slice category $d/(i^{op}/\bn)$ is non-empty and connected for each object $d = \{\bp,\bq,\si\} \in \wed_{\LA}/\bn$.  Again writing in terms
of $\SI$ and $\LA$ rather than their opposites, assume that $\si$ maps $j$ elements of  $\{1,\cdots,n\}$ to $\{1,\cdots,p\}$ and the remaining $k = n-j$ elements 
to $\{p+1,\cdots,p+q\}$.  Let
$$ \si(\{i_1, \cdots, i_j\}) \subset \{1, \cdots, p\} \ \ \text{and} \ \ \si(\{i_{j+1}\cdots, i_n\})\subset \{p+1,\cdots, p+q\},$$
where we label the $i$'s in increasing order.  Define $\si_1\colon \bj\rtarr \bp$ and $\si_2\colon \bk \rtarr \bq$ by
$$  \si_1(a) = \si(i_a)  \ \ \text{for} \ \ 1\leq a\leq j \ \ \text{and} \ \ \si_2(b) =
\si(i_{j+b})-p \ \ \text{for} \ \  1\leq b\leq k. $$
Let $\alpha$ be the  $(j,k)$-shuffle specified by $\al(a) = i_a  \ \ \text{for} \ \ 1\leq a\leq j+k. $
Then $n=j+k$ and $\si$ is the composite
\[ \xymatrix@1{ \bn \ar[r]^-{\alpha^{-1}} & \mathbf{j+k}  \ar[r]^-{\si_1\vee \si_2} &  \mathbf{p+q} \\} \]
In particular, if $\bp = \bj$ and $\bq = \bk$, we can decompose $\tau : \bn \rtarr \bn$ as $(\tau_1 \vee
\tau_2) \com \beta^{-1}$, where  $\tau_1: \bj \rtarr \bj$, $\tau_2: \bk \rtarr \bk$, and $\beta: \bn \rtarr \mathbf{j+k}$
is a  $(j,k)$-shuffle .

Then an object of $d/(i^{op}/\bn)$ is an object $\{\bj,\bk,\ta\}$ of  $\wed_{\SI}/\bn$ together with
injections $\mu\colon \bj\rtarr \bp$ and $\nu\colon \bk\rtarr \bq$ such that $(\mu \vee \nu)
  \com \ta = \si$. This implies that
$$\alpha = \beta, \ \ \mu\com \ta_1 = \si_1  \ \ \text{and} \ \ \nu\com \ta_2 = \si_2. $$
Since $\si$ determines $j$ and $k$, $\ta$ prescribes this object.  If $\ta'$
prescribes another such object, then there is a unique isomorphism  $\{\bj,\bk,\ta\}\rtarr  \{\bj,\bk,\ta'\}$ in $\wed_{\SI}/\bn$, namely 
$(\tau_1'\tau_1^{-1}, \tau_2'\tau_2^{-1}) \in \Sigma_j \times \Sigma_k$.  Thus $d/(i^{op}/\bn)$ is non-empty and connected.

It is illuminating to unravel this categorical proof. Let $\sD$ and $\sE$ be in $\LA^{op}[\sV]$.  For each $\bn$, we have the commutative square
\begin{equation}\label{wow}
 \xymatrix{  
\coprod_{j+k = n} \sD(\bj) \otimes \sE(\bk) \otimes \LA(\bn,\bn) \ar[r]^-{\pi} \ar[d]_{\subset} & (\sD\boxtimes_{\SI} \sE)(\bn)  \ar[d]^\iota  \\
\coprod_{p+q\geq n} \sD(\bp)\otimes \sE(\bq) \otimes \LA(\bn, \mathbf{p+q}) \ar[r]_-{\rh}
\ar@{-->}[ur]^{\tilde{\rho}} & (\sD\boxtimes_{\LA} \sE)(\bn)\\} 
\end{equation}
Here the left arrow is the evident inclusion of components and $\pi$ and $\rh$ are the canonical maps given by the definitions of our tensor products of functors. 
The essential point of the proof is to construct a map $\tilde{\rho}$ that makes the diagram commute and factors through $\rh$ to induce an inverse to $\io$.
To construct $\tilde{\rho}$ explicitly, observe that 
$\sD(\bp)\otimes \sE(\bq) \otimes \LA(\bn, \mathbf{p+q})$
is the coproduct over the injections
$\si \colon \bn \rtarr \mathbf{p+q}$ of its subobjects 
$\sD(\bp)\otimes \sE(\bq) \otimes \{\si\}$.
It suffices to define $\tilde{\rho}$ on each such component.
As explained above, we may factor $\sigma$ as $(\sigma_1 \vee \sigma_2)\com \al^{-1}$.
By definition, $\rh$ coequalizes the pair of maps 
\begin{equation}\label{ein}
\sD(\bp) \otimes \sE(\bq)\otimes \LA(\bj,\bp)\otimes \LA(\bk,\bq)\otimes  \LA(\bn,\bn) \rtarr\sD(\bj)\otimes \sE(\bk) \otimes\LA(\bn,\bn)
\end{equation}
and
\begin{equation}\label{zwei}
\sD(\bp) \otimes \sE(\bq)\otimes \LA(\bj,\bp)\otimes \LA(\bk,\bq)\otimes\LA(\bn,\bn)  \rtarr \sD(\bp)\otimes \sE(\bq) \otimes \LA(\bn, \mathbf{p+q}). 
\end{equation}
Here \autoref{ein} is the tensor of the evaluation map

{\small{\[ \xymatrix@1{
\sD(\bp) \otimes \sE(\bq)\otimes \LA(\bj,\bp)\otimes \LA(\bk,\bq) \iso \sD(\bp) \otimes  \LA(\bj,\bp)\otimes \sE(\bq) \otimes \LA(\bk,\bq)
\ar[r] & \sD(\bj)\otimes \sE(\bk)\\} \]
}}

\noindent
of the contravariant functor $\sD\otimes \sE$ and the identity map of $\LA(\bn,\bn)$ and \autoref{zwei} is the tensor of the identity map of
$\sD(\bp) \otimes \sE(\bq)$ with the map
\[ \xymatrix@1{ \LA(\bj,\bp)\otimes \LA(\bk,\bq)\otimes \LA(\bn,\bn) \ar[r]^-{\vee\otimes \id} & \LA(\bn,\mathbf{p+q})\otimes \LA(\bn,\bn)  \ar[r]^-{\com} 
& \LA(\bn,\mathbf{p+q}). \\} \]
Restrict the domains of  \autoref{ein} and \autoref{zwei} to the component
  \[ \sD(\bp) \otimes \sE(\bq)\otimes \{\si_1\} \otimes \{\si_2\} \otimes \{\al^{-1}\}.\]
Then \autoref{ein} gives an isomorphism
$$f_{\si}:\sD(\bp) \otimes \sE(\bq)\otimes \{\si_1\} \otimes \{\si_2\} \otimes \{\al^{-1}\} \to \sD(\bj)\otimes \sE(\bk)\otimes \{\al^{-1}\},$$
while \autoref{zwei} gives an isomorphism
$$g_{\si}: \sD(\bp) \otimes \sE(\bq)\otimes \{\si_1\} \otimes \{\si_2\} \otimes \{\al^{-1}\} \to
\sD(\bp)\otimes \sE(\bq) \otimes \{\si\}.$$
The composite
\[ \pi\com f_{\si}\com g_{\si}^{-1}\colon \sD(\bp)\otimes \sE(\bq) \otimes \{\si\} \rtarr (\sD\boxtimes_{\SI} \sE)(\bn) \]
gives the required component of $\tilde{\rho}$.  Letting $p$, $q$, and $\si\colon \bp\rtarr\bq$ vary, we obtain $\tilde{\rho}$ as the sum of these maps.  
In a more general but notationally simpler context, the categorical  proof of \cite[Lemma  8.34]{Riehl} essentially shows that $\tilde{\rho}$ factors through $\rh$ to give the
inverse $\io^{-1}$, as  we originally showed directly.
\end{proof}

\begin{rem}  Via the isomorphism $\io$, it follows that for $\LA$-sequences $\sD$ and $\sE$, the $\SI$-sequence $\sD\boxtimes_{\SI} \sE$  extends to a functor 
$\LA^{op}\rtarr \sV$.   In effect, the proof is constructing that extension, and then the intuition is that \autoref{coYon} kicks in to  give the isomorphism
\[  \sD\boxtimes_{\LA} \sE \iso  \int^{r\in \LA}  (\sD\boxtimes_{\SI} \sE)(r) \otimes \LA(-,r). \]
\end{rem}

\begin{rem} Formally, letting $i\colon \SI\rtarr \LA$ be the inclusion, we have the following commutative diagram, where $i^*$ is the evident forgetful functor.
\[
\xymatrix{
\sV_{\mI} \ar[r]^{u} \ar[d]_{i_0} & \sV  \ar@{=}[r]  \ar[d]^{i_0} & \sV \ar[d]^{i_0} \\
\LA^{op}[\sV]_{\mI_0} \ar[r]_-{u} & \LA^{op}[\sV] \ar[r]_{i^*} & \SI^{op}[\sV]\\}  \]
The functors $u$ forget the base maps, and they have left adjoints $(-)_+$ obtained by adjoining base objects.  Thus $X_+= X\amalg \mI$ 
for $X\in \sV$ and $\sD_+ = \sD\amalg \mI_0$ for $\sD\in \LA^{op}[\sV]$.  The functor $i^*$ has left adjoint $i_*$ given by left Kan extension along $i$. An easy formal argument shows that $i_*$ is strong symmetric monoidal, so that
\[ i_*\sD\boxtimes_{\LA} i_*\sE \iso i_*(\sD\boxtimes_{\SI} \sE) \]
for  $\sD,\sE\in \SI^{op}[\sV]$. It follows that $i^*$ is lax symmetric monoidal via the composite
$$\io \colon i^*\sD\boxtimes_{\SI}i^*\sE \rtarr  i^*i_*(i^*\sD\boxtimes_{\SI} i^*\sE) \iso i^*(i_*i^*\sD\boxtimes i_*i^*\sE) \rtarr i^*(\sD\boxtimes_{\LA}\sE)
$$
for $\sD,\sE\in \LA^{op}[\sV]$, where the unlabelled arrows are given by the unit and counit of the adjunction.  Deleting $i^*$ from the notation, inspection of definitions shows that this is the map $\io$ of \autoref{boxbox}.  The fact that it is an isomorphism should be viewed as an analog of the triviality that the identity map $\sV=\sV$ is strong symmetric monoidal.  Together with the observation that the adjunction $((-)_+,u)$ for the top row maps to the adjunction $((-)_+\com i_*,i^*\com u)$  for the bottom row, this gives substance to our analogy between the comparison of unbased and based spaces and the comparison of $\SI$-sequences and $\LA$-sequences.
\end{rem}

From \autoref{boxbox}, we obtain the following analog inductively.

\begin{thm}\label{loxlox}  For $\sE$ in  $\SI^{op}[\sV]$,
\[  \sE^{\boxtimes_{\SI} k}(\bn)  \iso \coprod_{j_1+ \cdots + j_k = n}  
\big( \sE(\bj_1) \otimes \ldots \otimes \sE(\bj_k) \big)  \otimes_{\SI_{j_1}\times \cdots \times \SI_{j_k}} \SI_n.  \]
For $\sE$ in $\LA^{op}[\sV]_{\mI_0}$, the natural map
\[  \io\colon  \sE^{\boxtimes_{\SI}k}(\bn) \rtarr  \sE^{\boxtimes_{\LA}k}(\bn)  \]
is an isomorphism for each $n$.
\end{thm}

In effect, when $\sE$ is the restriction of a $\LA$-sequence to $\SI$,  the proof of \autoref{loxlox} shows how to construct the  $\Lambda$-sequence structure of $\sE^{\boxtimes_{\LA}k}(\bn)$ detailed in \autoref{combtoo} from $\sE^{\boxtimes_{\SI}k}(\bn)$.

\section{The Kelly monoidal structures}\label{mon}

Day convolution is a stepping stone towards the product, denoted $\odot$, on
$\LA^{op}[\sV]_{\mI_0}$ that we really care about.  We call it the Kelly product
since it is an extension of the Kelly product (alias composition product) on $\SI^{op}[\sV]$,
which we denote by $\odot_{\Sigma}$.

 As noted in \autoref{prel}, the left Kan extensions of \autoref{Kan1} and \autoref{Kan2} are tensor products of functors.  Specializing \autoref{coendtoo}, we have the tensor product
 \[  \sD\otimes_{\LA} \sE \]
of a contravariant functor $\sD\colon \LA \rtarr \sV$ and a covariant functor $\sE\colon \LA \rtarr \sV$.  It is the coequalizer in $\sV$ displayed in the diagram 
\[ \xymatrix@1{ 
\coprod_{(\bi,\bj)}  \sD(\bj)\otimes \LA(\bi,\bj) \otimes \sE(\bi) \ar@<1ex>[r]^-{{ev} \otimes \id}  \ar@<-1ex>[r]_-{\id\otimes {ev} }
& \coprod_{\bk} \sD(\bk)\otimes \sE(\bk) \ar[r] &   \sD\otimes_{\LA} \sE}, \]
where  we have written $ev$ for the evaluation maps of the functors $\sD$ and $\sE$. Schematically, in terms of elements where that makes sense,
$\la^*(d)\otimes e\sim d\otimes \la_*(e)$ for a map $\la\colon \bf i\rtarr \bf j$ in $\LA$ and elements $d\in \sD(\bj)$ and $e\in \sE(\bi)$. 

Observe that $\otimes_{\LA}$ can be viewed as a quotient of $\otimes_{\SI}$ that is obtained by adding the injections to the permutations. In practice, the quotient is obtained by base object identifications.  Most of the topological literature on operads, starting with \cite{MayGeo}, focuses on this tensor product over $\LA$, but most of the algebraic literature focuses on the analogous tensor product over $\SI$.   As we shall recall, the following special case, first formulated in \cite{CMT}, generalizes the construction of the monad in  $\sV_{\mI}$ associated to a unital operad.  Recall \autoref{Xstar}.

\begin{defn}\label{DXdefn}  For $\sD\in \LA^{op}[\sV]$,  define a functor $\bD\colon \sV_{\mI} \rtarr \sV_{\mI}$ by
\[ \bD X = \sD\otimes_{\LA} X^{\otimes *}. \]
Observe that $\bD\mI = \sD(\mathbf{0})$ since injection identifications glue the  terms $\sD(k) \otimes \mI^{\otimes k}$ 
for $k>0$ to $\sD(\mathbf{0})\otimes \mI = \sD(\mathbf{0})$.
\end{defn}

Thinking of $\bD X$ as a kind of product of $\sD$ and $X$, we obtain the Kelly product as a generalized analog. Recall \autoref{Estar}.

\begin{defn}\label{Kelly0LA}  Define the Kelly product $\odot$ on  $\Lambda^{op}[\sV]_{\mI_0}$ by
\[   \big ( \sD  \odot \sE\big )(\bn) = \sD   \otimes_{\Lambda} \sE^{\boxtimes *}(\bn)\]
for $\LA$-sequences $\sD$ and $\sE$.  Thus its value on $\bn$  is the coend 
\[  \int^{k} \sD(\bk)\otimes \Ek(\bn). \] 
More explicitly,  $\big (\sD  \odot \sE \big )(\bn)$ is the coequalizer 
\[ \xymatrix@1{  \coprod_{(i,j)}  \sD(\bj) \otimes  \LA(\bi,\bj) \otimes \sE^{\boxtimes i}(\bn)
 \ar@<1ex>[r]^-{{ev} \otimes \id}  \ar@<-1ex>[r]_-{\id\otimes {ev} }
 & \coprod_\bk \sD(\bk)\otimes \Ek(\bn)  \ar[r] &   \big( \sD\odot \sE\big)(\bn). \\ } \]
 In particular, $\big( \sD\odot \sE\big)(\mathbf{0}) = \bD\sE(\mathbf{0})$.
The base maps $\et$ of $\sD$ and $\sE$ induce the base map of $\sD  \odot \sE$ as the composite
$$\xymatrix@1{\mI \ar[r]^-{\et} & \sD(\mathbf{0}) = \bD\mI \ar[r]^-{\bD\et} &  \bD\sE(\mathbf{0})  = \big ( \sD  \odot \sE\big )(\mathbf{0})\\}.$$
Observe that if $\sD(\mathbf{0})=\mI$ and $\sE(\mathbf{0})= \mI$, then these are identifications, so that $\odot$ restricts to a product on unital $\LA$-sequences.  

The Kelly product $\odot$ on $\SI^{op}[\sV]$ is defined by replacing $\LA$ by $\SI$ and ignoring base maps in this definition.  We write $\odot_{\LA}$ and $\odot_{\SI}$ when necessary for clarity, but  $\odot = \odot_{\LA}$ is the default.
\end{defn}

Although we are mainly interested in the case when $\sD$ and $\sE$ are unital, we shall make use of the  following comparison between the functor
$i_0$ of \autoref{notn3}, which gives based but not unital choices for $\sE$, and the Kelly product. 

\begin{lem}\label{iodot}
If $\sD\in \LA^{op}[\sV]_{\mI_0}$ and  $X\in \sV_{\mI}$, then $\sD \odot i_0 X$ is isomorphic to $i_0 \bD X$.
\end{lem}
\begin{proof}
Since $(i_0X)(\bk) = \emptyset$ for $k>0$, $\sD \odot i_0 X$ and $i_0 \bD X$ only take nontrivial values at level $0$, and 
$(i_0X)^{\boxtimes *}(\mathbf{0}) = X^{\otimes *}$ by comparison of definitions.  Therefore
$$(\sD \odot i_0 X)(\mathbf{0}) = \bD X = (i_0 \bD X)(\mathbf{0}).   \qedhere $$  
\end{proof}

In analogy with $\mI_0$, we need the following $\LA$-sequence $\mI_1$.

\begin{notn}\label{notn4} Define a $\LA$-sequence $\mI_1$ by letting $\mI_1(\mathbf{0}) = \mI_1(\mathbf{\mathbf{1}}) = \mI$ and $\mI_1(\bn) = \emptyset$ for $n>1$.  The injection $\mathbf{0} \rtarr \mathbf{1}$  induces the identity map 
 $\mI_1(\mathbf{1})\rtarr \mI_1(\mathbf{0})$.  In analogy with \autoref{notn3},  define $\LA^{op}_{\mI_1}$ to be the category of 
 $\LA$-sequences under $\mI_1$ and define a functor $i_1\colon \sV_{\mI}\rtarr \LA^{op}[\sV]_{\mI_1}$ by letting 
 $i_1 X = X\otimes \mI_1$, 
 where $\otimes$ is defined levelwise.  Thus $i_1X = X$ at levels $0$ and $1$ and $i_1X = \emptyset$ at levels $n >1$.  \end{notn}

 In the rest of this section, we shall prove the following result. 

\begin{thm}\label{boxtimes}  Under the Kelly product $\odot$, the category $\LA^{op}[\sV]_{\mI_0}$ is monoidal with unit object $\mI_1$. 
Therefore the subcategory $\LA^{op}_{\mI}[\LA]$ is also monoidal under $\odot$ with unit object $\mI_1$.  Moreover, $\LA^{op}[\sV]_{\mI_1}$
is also monoidal with unit $\mI_1$ under $\odot$, and the functor 
$$i_1\colon (\sV_{\mI},\otimes, \mI) \rtarr (\LA^{op}[\sV]_{\mI_1},\odot,\mI_1)$$
is strong monoidal.
\end{thm}

\begin{rem}\label{odotunits}  For $\SI$-sequences $\sD$ and $\sE$, the $n$th component 
$\big(\sD  \odot_{\SI} \sE\big)(\bn)$  is 
\begin{align*}  
\big(\sD\otimes_{\SI} \sE^{\boxtimes *}\big)(\bn)  &= 
   \coprod_{k\ge 0} \sD  (\bk) \otimes_{\Sigma_k} \Ek(\bn) \\
     & = \coprod_{k\ge 0} \sD(\bk) \otimes_{\Sigma_k} \left( \coprod_{j_1+\cdots + j_k=n} (\sE (\bj_1) \otimes \ldots \otimes \sE (\bj_k)) \otimes_{\Sigma_{\bj_1} \times \ldots \times \Sigma_{j_k}} \Sigma_n\right) \\
     & =  \coprod_{k\ge 0}\coprod_{j_1+\cdots + j_k=n} \sD  (\bk) \otimes_{\Sigma_k}  \Big( (\sE (\bj_1) \otimes \ldots \otimes \sE (\bj_k)) \otimes_{\Sigma_{j_1} \times \ldots \times \Sigma_{j_k}} \Sigma_n \Big).
\end{align*}
Remember again that there are no maps $\bm \rtarr \bn$ in $\SI$ when $m \neq n$.   In line with our unbased versus based comparison, the unit for the resulting monoidal structure on $\SI^{op}[\sV]$ is not $\mI_1$  but rather its variant with the value $\emptyset$ at $0$. Conceptually,
the respective units are 
$$\mI_1^{\SI} =  \SI^{op}(\mathbf{1},-) \ \ \text{and}\ \   \mI_1^{\LA} =  \LA^{op}(\mathbf{1},-),$$
where we again think of $Set$ as embedded in $\sV$ via the functor $S \mapsto \mI[S]$.  We have an evident comparison map of $\SI$-sequences 
$\nu\colon \mI_1^{\SI} \rtarr \mI_1^{\LA}$, and it is a quick exercise to check that $\mI_1^{\SI} \otimes_{\SI} \LA = \mI_1^{\LA}$.
\end{rem}

In view of \autoref{loxlox},  we can identify $\boxtimes = \boxtimes_{\LA}$ with $\boxtimes_{\SI}$.  Therefore $\odot_{\LA}$ is the composite of the construction of  $\odot_{\SI}$ just described and base object identifications.   Formally, we have the forgetful functor $i^*\colon \LA^{op}[\sV]_{\mI_0} \rtarr \SI^{op}[\sV]$, and we can view $\sD \odot  \sE$ as a quotient of $(i^*D \odot_{\SI}  i^{\boxtimes *}E)$.   Explicitly, the $n^{th}$ component $\big(\sD  \odot_{\LA} \sE\big)(\bn)$ is the quotient of $\big(\sD  \odot_{\SI} \sE\big)(\bn)$ obtained by iterating the following procedure to deal with injections 
$\si_i\colon \mathbf{k} \rtarr \mathbf{k+1}$ for $1\leq i\leq k+1$.  For each sequence $(j_1,\ldots, j_k)$ with $j_1 + \cdots + j_k = n$, we attach
$$  \sD(\mathbf{k+1}) \otimes \sE (\bj_1) \otimes \cdots \otimes \mI \otimes \cdots \otimes \sE (\bj_k) \otimes_{\Sigma_{j_1} \times \cdots \times \SI_0\times \cdots \times \Sigma_{j_k}} \Sigma_n,$$ 
where $\mI$ and $\SI_0 = \ast$  appear in the $i$th place,  to 
$$\sD (\bk)\otimes \sE (\bj_1) \otimes \ldots \otimes \sE (\bj_k) \otimes_{\Sigma_{j_1} \times \ldots \times \Sigma_{j_k}} \Sigma_n $$
via coequalization of the maps
\[ \xymatrix{\sD(\mathbf{k+1})\otimes \sE(\bj_1) \otimes \cdots \otimes \sE(\bj_k)
\otimes_{\Sigma_{j_1} \times \cdots \times \Sigma_{j_k}} \Sigma_n
\ar[d]^{ \id\otimes\si_i } \\
 \sD(\mathbf{k+1})\otimes  \sE(\bj_1) \otimes \cdots \otimes \sE(\mathbf{0}) \otimes \cdots \otimes \sE(\bj_k)\otimes_{\Sigma_{j_1} \times \cdots \times \SI_0\times \cdots \times \Sigma_{j_k}} \Sigma_n  \\} \]
obtained by inserting $\et\colon \mI \rtarr \sE(0)$ in the $i$th position and
\[ \xymatrix{
\sD(\mathbf{k+1}) \otimes \sE(\bj_1) \otimes \cdots \otimes \sE(\bj_k) \otimes_{\Sigma_{j_1} \times \ldots \times \Sigma_{j_k}} \Sigma_n
\ar[d]^{\si_i\otimes \id} \\
 \sD(\bk)\otimes \sE(\bj_1) \otimes \ldots \otimes \sE(\bj_k) \otimes_{\Sigma_{j_1} \times \ldots \times \Sigma_{j_k}} \Sigma_n.  \\}  \] 
 
To prove the associativity of $\odot$, we shall need the following result about the Day product $\boxtimes$.  It is an analog of \cite[Lemma 3.1]{Kelly0}. 
  
 \begin{prop}\label{key}
Let  $\sE\in \LA^{op}[\sV]_{\mI_0}$.  Then the functor
 $$ - \odot \sE:  (\LA^{op}[\sV]_{\mI_0}, \boxtimes, \mI_0) \rtarr (\LA^{op}[\sV]_{\mI_0}, \boxtimes, \mI_0)$$
 is strong symmetric monoidal. Therefore, there is a natural isomorphism 
 $$\sD^{\boxtimes *} \odot \sE \cong (\sD \odot \sE)^{\boxtimes *}$$
 for $\sD\in \LA^{op}[\sV]_{\mI_0}$.
 \end{prop}
\begin{proof}
Let $\sD,\sD'\in \LA^{op}[\sV]_{\mI_0}$. We must prove that
$$(\sD \boxtimes \sD') \odot \sE \iso (\sD \odot \sE ) \boxtimes (\sD' \odot \sE)$$
and that $\mI_0\boxtimes \sE \iso \mI_0$. For the first, we  compute directly.  While we could use \autoref{loxlox} to help see the conclusion, that is not necessary. For any $\bn \in \LA$,
\begin{align*}
&\quad \quad ((\sD \boxtimes \sD') \odot \sE)(\bn) \\
& =   \int^\bm (\sD\boxtimes \sD')(\bm) \otimes (\sE ^{\boxtimes m}(\bn))\\
&= \int^\bm \left(\int^{\bk_1, \bk_2} \sD(\bk_1) \otimes \sD'(\bk_2) \otimes  \LA(\bm, \bk_1 + \bk_2)
\right) \otimes  (\sE ^{\boxtimes m}(\bn)) \\
& \iso \int^{\bk_1, \bk_2} \sD(\bk_1) \otimes\sD'(\bk_2)  \otimes \int^\bm \LA(\bm, \bk_1 + \bk_2)\otimes  \sE ^{\boxtimes m}(\bn)\\
& \iso  \int^{\bk_1, \bk_2} \sD(\bk_1) \otimes \sD'(\bk_2)  \otimes \sE ^{\boxtimes k_1 +k_2}(\bn) \\
& = \int^{\bk_1, \bk_2}  \sD(\bk_1) \otimes \sD'(\bk_2)  \otimes \left((\sE^{\boxtimes k_1} \boxtimes \sE^{ \boxtimes k_2})(\bn)\right) \\
& = \int^{\bk_1, \bk_2}  \sD(\bk_1) \otimes \sD'(\bk_2)  \otimes \left(\int^{\bj_1, \bj_2} \sE^{\boxtimes k_1} (\bj_1) \otimes \sE^{ \boxtimes k_2} (\bj_2) \otimes  \LA(\bn, \bj_1 + \bj_2) \right) \\
& \iso \int^{\bj_1, \bj_2} \left(\int^{\bk_1} \sD (\bk_1) \otimes \sE^{\boxtimes k_1} (\bj_1)\right) \otimes \left(\int^{\bk_2} \sD'(\bk_2) \otimes \sE^{ \boxtimes k_2} (\bj_2)\right)\otimes \LA(\bn, \bj_1 + \bj_2)\\
& =  \int^{\bj_1, \bj_2} (\sD \odot \sE)(\bj_1) \otimes (\sD' \odot\sE) (\bj_2) \otimes \LA(\bn, \bj_1 + \bj_2) \\
& = \big( (\sD \odot \sE) \boxtimes (\sD' \odot \sE)\big)(\bn) 
\end{align*}

The equalities are given by expansion of definitions.  For the second isomorphism, the Yoneda lemma dual to \autoref{coYon} gives an isomorphism 
$$\sE^{\boxtimes k_1+ k_2}  \cong \int^{\bm \in \LA}  \LA(\bm, \bk_1 + \bk_2)\otimes \sE^{\boxtimes m} $$
for each pair $(k_1,k_2)$.
The first and third isomorphisms use the  categorical Fubini Theorem of \cite[\S IX.8]{Mac}, together with a reordering of variables transposition, which uses $(\LA \times \LA)^{op} \times \LA \times \LA \cong \LA^{op} \times \LA \times \LA^{op} \times \LA$.  For the unit isomorphism,  the definition of $\mI_0$ implies that $\mI_0 \odot \sE(\mathbf{0}) \iso \mI$ and that $\mI_0 \odot \sE(\bn) = \emptyset$ for $n > 0$.  We should check that these isomorphisms satisfy the unit, associativity, and symmetry coherence conditions specified by MacLane \cite[\S XI]{Mac}, but we  leave these straightforward verifications to the reader. Taking $\sD' = \sD^{\boxtimes k}$, we see inductively that the second statement of the proposition follows from the first.
\end{proof}

We shall later need the following consequence of \autoref{key}.

\begin{prop}\label{key2}  Let $X\in \sV_{\mI}$.  Then the functor
 $$ - \otimes_{\LA} X^{\otimes *}:  (\LA^{op}[\sV]_{\mI_0}, \boxtimes, \mI_0) \rtarr (\sV_{\mI}, \otimes, \mI)$$
 is strong symmetric monoidal. Therefore, there is a natural isomorphism 
 $$\sD^{\boxtimes *} \otimes_{\LA} X^{\otimes *} \cong (\sD \otimes_{\LA} X^{\otimes *})^{\otimes *}$$
 for $\sD\in \LA^{op}[\sV]_{\mI_0}$.
 \end{prop}
\begin{proof}  By \autoref{Kelly0LA}, the functor $ - \otimes_{\LA} X^{\otimes *}$ is the composite of $p_0$ and $ - \odot i_0X$, which are strong symmetric monoidal by \autoref{strongp} and \autoref{key}.
\end{proof}
 
\begin{proof}[Proof of \autoref{boxtimes}]
For convenience, following \cite{Kelly0}, we write
$${\mI}_1 \cong \mI \otimes \LA(-, \mathbf{1}). $$
To see that ${\mI}_1$ is a left unit, we observe by \autoref{coYon} that 
$${\mI}_1 \odot \sD = \int^{\bk} {\mI}_1(\bk)  \otimes \sD^{\boxtimes k}(-)\cong \int^{\bk} \mI\otimes \LA(\bk,\mathbf{1})  \otimes \sD^{\boxtimes k}(-) \cong \sD^1. $$
To see that it is also a right unit, we first compute
\begin{align*}
{\mI}_1^{\boxtimes k} & = \int^{\bj_1, \ldots, \bj_k} (\LA(\bj_1, \mathbf{1}) \otimes \mI) \otimes \ldots \otimes (\LA(\bj_k,\mathbf{1}) \otimes \mI) \otimes  \LA(-, \bj_1 + \ldots + \bj_k) \\
& = \LA(-, \mathbf{1} + \ldots +\mathbf{1}) \otimes  \mI \\
& = \LA(-,\bk) \otimes \mI.
\end{align*}
Using \autoref{coYon} again, this implies  that
$$\sD \odot {\mI}_1 = \int^{\bk} \sD(\bk) \otimes {\mI}_1^{\boxtimes k}(-) \cong \int^{\bk}  \sD(\bk) \otimes \LA(-,\bk) \otimes \mI \cong\sD.$$
The associativity follows from \autoref{key}:
\begin{align*}
\sD \odot (\sE \odot \sF) &=\sD  \otimes_{\LA} (\sE \odot \sF)^{\boxtimes *}  \\
& \iso \sD \otimes_{\LA} (\sE^{\boxtimes *} \odot\sF)\\
& = \sD \otimes_{\LA} (\sE^{\boxtimes *} \otimes_{\LA} \sF^{\boxtimes *})\\
& \iso (\sD\otimes_{\LA} \sE^{\boxtimes *}) \otimes_{\LA} \sF^{\boxtimes *} \\
& = (\sD \odot \sE) \odot \sF.
\end{align*}

Finally, consider the statement regarding $i_1$. As the category under the unit,
$\LA^{op}[\sV]_{\mI_1}$ is certainly monoidal. 
Observe that $i_1X =  X \otimes \LA(-, \mathbf{1})$ for $X\in \sV_{\mI}$.
Writing  $X\otimes \sE$ more generally for the levelwise tensor of $X$ with $\sE \in \LA^{op}[\sV]_{\mI_1}$, 
we have
\begin{align*}
i_1X \odot\sE(-) & = \int^{\bk} (i_1X)(\bk) \otimes \sE^{\boxtimes k}(-)  \\
& \cong \int^{\bk}  X \otimes \LA(\bk, \mathbf{1}) \otimes \sE^{\boxtimes k} (-)  \\
& \cong X \otimes  \int^{\bk}   \LA(\bk, \mathbf{1}) \otimes \sE^{\boxtimes k}(-)    \\
& \cong X \otimes \sE (-). 
\end{align*}
Taking $\sE = i_1(Y)$ and observing that $X \otimes i_1(Y) \cong i_1(X \otimes Y)$, we have
\begin{equation*}
 i_1(X) \odot i_1(Y) \cong i_1(X \otimes Y).   \qedhere
\end{equation*}
\end{proof}

Now that we have $\odot$, we can use \autoref{key} to give a conceptual reinterpretation of \autoref{boxbox}.  For a monoid 
$\sE$ in the  monoidal category $(\LA^{op}[\sV],\odot,\mI_0)$, a right $\sE$-module is an object $\sD$  in $\LA^{op}[\sV]$ with a right action $ \sD\odot \sE \rtarr \sD$ satisfying the standard axioms, as in \autoref{module} below. 
\autoref{key} has the following immediate consequence. 
 \begin{cor}
\label{rightmodule} For a monoid $\sE$ in $(\LA^{op}[\sV],\boxtimes,\mI_0)$, the category of right $\sE$-modules inherits a symmetric monoidal product $\boxtimes$ from $\boxtimes$ on $\LA^{op}[\sV]_{\mI_0}$.
 \end{cor}
 \begin{proof}
   Let $\sD$ and $\sD'$ be right $\sE$-modules. Then the right $\sE$-module
   structure on $\sD \boxtimes \sD'$ is given by 
\begin{equation*}
(\sD \boxtimes \sD') \odot \sE \cong (\sD \odot \sE) \boxtimes (\sD' \odot \sE) \rtarr \sD \boxtimes \sD'.
\end{equation*}
Routine diagram chasing proves the result.
 \end{proof}
 
Recall that $\mI_1$ is the $\Lambda$-sequence defined  in \autoref{notn4}.
\begin{thm}(\cite[5.1.8]{Fresse0})
  \label{LA-SI}
Viewing $\mI_1$ as a $\Sigma$-sequence, the category $\Lambda^{op}[\sV]$ is isomorphic to the category of right
  $\mI_1$-modules in $(\Sigma^{op}[\sV], \odot_{\Sigma})$.
\end{thm}
\begin{proof}[Proof sketch]  
Taking $C= \mI$ in Fresse's \cite[5.1.1 and 5.1.7]{Fresse0}, his $*_{C}$ is our $\mI_1$.  Comparing combinatorics shows that  his description of a right $\mI_1$-module structure on a symmetric sequence $\sD$ gives $\sD$ the structure of a $\LA$-sequence, and conversely.
\end{proof}

\begin{rem} Consider $\LA$-sequences $\sD$ and $\sE$.  Viewing them as right $\mI_1$-modules in 
$(\Sigma^{op}[\sV], \odot_{\Sigma})$, the $\SI$-variant of \autoref{rightmodule}, says that
$\sD\boxtimes_{\SI}\sE$ is again a right $\mI_1$-module and thus a $\LA$-sequence. The content of \autoref{boxbox} is that this $\LA$-sequence is precisely the Day convolution product 
$\sD\boxtimes_{\Lambda}\sE$.
\end{rem} 

\section{The redefinition of based and unital operads}\label{oper}

With the monoidal category $\LA^{op}[\sV]_{\mI_0}$ in place, we can give our new definition of based and unital operads in $\sV$.

\begin{defn}\label{new} A based operad in $\sV$ is a monoid in the monoidal category $\LA^{op}[\sV]_{\mI_0}$.
 A unital operad in $\sV$ is a monoid in the monoidal category $\LA^{op}_{\mI}[\sV]$.
\end{defn}

We shall not repeat the complete ``classical'' definition of an operad $\sC$ in $\sV$ \cite{Kelly0, MayOp1, MayGeo}, but we denote its structural maps by 
$$\ga\colon \sC(k)\otimes \sC(j_1) \otimes \cdots \otimes \sC(j_k) \rtarr \sC(j),$$
where $j = j_1+ \cdots + j_k$. These maps are required to be associative, unital, and equivariant in the sense prescribed in \cite{MayGeo, MayOp1}.    Of course, we say that $\sC$ is unital if $\sC(0) = \mI$.  With the classical definition, we have the following standard observation.

\begin{lem}\label{Cstar}  If $\sC$ is a based operad, then the $\sC(k)$ give the values on objects of a $\LA$-sequence.
\end{lem}
\begin{proof} The object $\sC(k)$ comes with a right action of $\SI_k$.  We have the base object map $\et\colon \mI\rtarr\sC(0)$, and $\sC(1)$ comes with a unit map $\io\colon \mI \rtarr \sC(1)$.    We define $\si_i\colon \sC(k)\rtarr \sC(k-1)$, $1\leq i\leq k$, to be the composite
\[
\xymatrix{\sC(k) \iso   \sC(k) \otimes \mI^{\otimes k} \ar[d]^-{\id\otimes \io^{\otimes i-1} \otimes \et \otimes \io^{\otimes k-i }} \\
\sC(k) \otimes \sC(1)^{\otimes i-1} \otimes \sC(0)\otimes \sC(1)^{\otimes k-i} \ar[d]^{\ga} \\
\sC(k-1).\\} 
\]
\end{proof}

The following restatement of \autoref{FRP} is an instance of the first author's favorite kind of comparison result. 

\begin{thm}\label{taut} The category of based operads in the new sense is isomorphic to the category of based operads in the classical sense.  
The category of unital operads in the new sense is isomorphic to the category of unital operads in the classical sense. 
\end{thm}

If $\sC$ is a monoid in $\LA^{op}[\sV]_{\mI_0}$,  then its product is given by structure maps
\[  \ga\colon \sC(k)\otimes \sC(j_1)\otimes \cdots \otimes\sC(j_k) \rtarr \sC(j_1 + \cdots + j_k).\]
Restricting to $\SI$, a direct comparison of definitions shows that to give a well-defined monoid structure, these must satisfy the precise associativity, unity, and equivariance constraints  specified in \cite{Kelly0, MayGeo, MayOp1}, and th

\begin{rem}\label{distinct}  There is a crucial logical distinction to be made between \autoref{Kelly0} and \autoref{taut}.  In the classical definition of an operad, one starts with a given $\SI$-sequence.   That is, the data one starts with is a set of objects $\sC(j)$ with given $\SI_j$-actions.  One then builds in structure maps on top of that.  In contrast, for based operads, the action of the injections in $\LA$ is given by the structure maps $\ga$ of the operad via \autoref{Cstar}, and their relationship to the rest of the operad structure is built into the unit and associativity axioms of the operad.    Thus, in interpreting \autoref{taut}, if we start with a given $\LA$-sequence, we must restrict attention to those operads whose underlying $\LA$-sequences, as built from the operad structure maps and the base map, coincide with the $\LA$-sequence that we start with.  This distinction loses force when we restrict attention to unital operads.  There, the definition of the underlying $\LA$-sequence is often clear without prior knowledge of $\ga$.
\end{rem}
 
\begin{proof}[Proof of \autoref{taut}]
For $\sD,\sE\in \LA[\sV]_{\mI_0}$, \autoref{loxlox} shows that the quotient map  
$q\colon \sD\otimes_{\SI}\sE \rtarr \sD \otimes_{\LA} \sE$   
that we have in general restricts to a quotient map 
$q\colon \sD\odot_{\SI}\sE \rtarr \sD \odot_{\LA} \sE$. 
We shall use that to deduce \autoref{taut} from \autoref{Kelly0}.  

With the understanding of \autoref{distinct}, let $\sC$ be a $\LA$-sequence with base map $\et\colon \mI\rtarr \sC(0)$.  Forgetting from $\LA$ to $\SI$,  a classical operad structure on the symmetric sequence $\sC$ is the same as a monoid structure on $\sC$.  Now consider based operad structures on $\sC$ with the given underlying $\LA$-sequence. The claim is that these are the same as monoid structures on the $\LA$-sequence $\sC$. This means that  the unit and product of the monoid in $\SI^{op}[\LA]$ factor as composites
\begin{equation}
\xymatrix@1{\mI_1^{\SI} \ar[r]^-{\nu} &  \mI^{\LA}_1 \ar[r]^-{\et} & \sC} \ \ \text{and}  \ \ 
\xymatrix@1{\sC\odot_{\SI}\sC  \ar[r]^-{q} &  \sC\odot_{\LA} \sC \ar[r]^-{\mu} & \sC.}
\end{equation}
The  map $\nu$ is the canonical comparison of unit $\SI$-sequences from \autoref{odotunits}. The map $\et$ is given by the base map 
$\mI \rtarr \sC(0)$ at level $0$, the unit map $\io \colon \mI\rtarr \sC(1)$ at level $1$, and the unique map $\emptyset \rtarr \emptyset$ at levels $n > 1$.  A unit condition in the definition of an operad gives that the triangle commutes in the following diagram

\[ \xymatrix{
\mI \ar[r]^-{\et}  \ar[d]_{\io} & \sC(0) \ar[r]^-{\iso} & \mI\otimes \sC(0)  \ar[dr]^{\iso} \ar[d]_{\io\otimes \id} & \\
\sC(1) \ar[r]_-{\iso} & \sC(1)\otimes \mI \ar[r]_-{\id\otimes \et} & \sC(1)\otimes \sC(0)  \ar[r]_-{\ga} & \sC(0). \\}
\]

By comparison with \autoref{Cstar}, this shows that  $\et$ is a map of $\LA$-sequences, and it is clear that $\et\com \nu$ is the unit of the monoid 
$\sC$ in $\SI^{op}[\sV]$.  Via \autoref{Cstar} and associativity conditions in the definition of a operad, there is  a unique map $\mu$ of $\LA$-sequences such that $\mu\com q$ is the product of the monoid $\sC$ in $\SI^{op}[\LA]$.  The associativity and unit conditions ensuring that $\mu$ and $\et$ give $\sC$ a structure of monoid in $\LA^{op}[\sV]_{\mI_0}$ are inherited from the monoid conditions for $\sC$ in $\SI^{op}[\LA]$.
\end{proof}

\section{A comparison of monoids and monads}\label{mmcomp}

We can now recast the specification of a monad in $\sV_{\mI}$ associated to a unital operad $\sC$ in terms of the Kelly product $\odot$.  Recall \autoref{Xstar}  and \autoref{DXdefn}.  We could work with based operads in $\LA^{op}[\sV]_{\mI_0}$, but we prefer to focus on
 $\LA^{op}_{\mI}[\sV]$. We have the following fundamental comparison.  

\begin{defn}\label{odot}  For $\sD, \sE\in \Lambda^{op}_{\mI}[\mathcal{V}]$, define  $\bD\odot \bE$ to be the functor on 
$\sV_{\mI}$ that \autoref{DXdefn} associates to $\sD\odot \sE \in \Lambda^{op}_{\mI}[\mathcal{V}]$.  That is, for $X\in \sV_{\mI}$,
\[   (\bD\odot\bE)(X) = (\sD \odot \sE) \otimes_{\Lambda} X^{\otimes *}.\]
\end{defn}

\begin{thm}
\label{equivcomp}
Let $\sD, \sE\in \Lambda^{op}_{\mI}[\mathcal{V}]$ and  $X\in \sV_{\mI}$.  Then the functors $\bD\odot \bE$ and $\bD\circ \bE = \bD\bE$ are isomorphic: 
$$(\bD \odot \bE)(X) \cong (\bD \bE)(X).$$
\end{thm}
\begin{proof}
It follows from the associativity of $\otimes_{\LA}$  and  \autoref{key2} that
\begin{align*}
(\sD \odot \sE) \otimes_{\Lambda} X^{\otimes *} 
 & = (\sD \otimes_{\Lambda} \sE^{\boxtimes *})  \otimes_{\Lambda} X^{\otimes *}  \\
 & \cong \sD \otimes_{\Lambda}( \sE^{\boxtimes *} \otimes_{\Lambda} X^{\otimes *})  \\
 & \cong \sD \otimes_{\Lambda} ( \sE \otimes_{\Lambda} X^{\otimes *})^{\otimes *}   \\
 & = \sD \otimes_{\Lambda} (\bE X)^{\otimes *}   \qedhere  \\
 \end{align*}   

 \end{proof}
We can now see that the monads associated to operads arise from their structures as monoids.  We state the result in the unital case, but the based case works the same way.

\begin{thm} Let $\sC$ be a unital operad.  Then  $\bC$ is a monad on the category $\sV_{\mI}$.
\end{thm}
\begin{proof}  We view $\sC$ as a monoid in $\LA^{op}_{\mI}[\sV]$.  The product $\mu\colon \bC\bC \rtarr \bC$ is  the functor 
\[    \bC\bC \iso \bC\odot \bC  \rtarr \bC \] 
induced by the product on $\sC$.  Observe that the functor $\bI_1$ on $\sV_{\mI}$ that \autoref{DXdefn} associates to 
$\mI_1$ is isomorphic to the identity.   The unit $\et\colon \id \rtarr \bC$  is the functor
\[  \id \iso \bI_1 \rtarr \bC \]
induced by the unit $\mI_1\rtarr \sC$ of $\sC$.  Diagram chases from the monoid axioms on $\sC$ show that they imply the monad axioms on $\bC$.
\end{proof}

Of course, we can compare definitions to see that this is the same monad structure that was defined in \cite{MayGeo, MayOp1}.  The classical definition of an action $\tha$ of $\sC$ on an object $X$ of $\sV_I$ is equivalent to giving an 
action $\tha\colon \bC X \rtarr X$ of the monad $\bC$ on $X$.  We can give a monoidal translation, but it will help to
digress a bit.  The notions of a left and a right module over a monoid in any monoidal category are standard, but they may seem unfamiliar in the present context of a monoid $\sC$ in $\LA^{op}_{\mI}[\sV]$.   We record the definition of a right module.

 \begin{defn}\label{module}
A right $\sC$-module $(\sM, \rho)$ is a functor $\sM \in \Lambda^{op}_{\mI}[\mathcal{V}]$ equipped with a right action 
$\rho: \sM \odot \sC \rtarr \sM$ such that the following diagrams commute: 
\[
\begin{tikzcd}
\sM \odot {\mI}_1 \arrow[rd, "\cong"', no head] \arrow[r, "id \odot \eta"] & \sM \odot \sC \arrow[d, "\rho"] \\
                                  & \sM               
\end{tikzcd} \quad \quad
\begin{tikzcd}
\sM \odot \sC \odot\sC  \arrow[r, " \rho \odot id"] \arrow[d, "id \odot \gamma  "'] & \sM \odot
 \sC \arrow[d, "\rho"] \\
\sM \odot \sC  \arrow[r, "\rho"']                & \sM             
\end{tikzcd}
\]
where $\eta\colon \mI_1 \rtarr \sC$ is the unit and $\gamma: \sC \odot \sC \rtarr \sC$ is the product of $\sC$.
\end{defn}

Of course, we have the symmetric definition of a left $\sC$-module $\sN$.  One obvious example is $\sN= \sC$.  However, we are interested in $\sC$-algebras.  Recall that, in monadic form, an action of $\sC$  on an object $X$ of $\sV_{\mI}$ is given by a map $\tha\colon \bC X \rtarr X$ such that the following diagrams commute in $\sV_{\mI}$.
\[\xymatrix{
 \bI_1X   \ar[r]^-{\et} \ar[dr]_{\iso} & \bC X \ar[d]^{\tha} \\
 & X}
 \ \ \ \ \ \ \
 \xymatrix{
 \bC\bC X  \ar[d]_{\bC\tha} \ar[r]^-{\mu}  & \bC X  \ar[d]^{\tha} \\
 \bC X  \ar[r]_-{\tha}  &  X \\} \]
By \autoref{equivcomp}, the upper left corner of the second diagram is isomorphic to $(\bC\odot
\bC)(X)$ and $\mu$ is induced by the product $\ga$ on $\sC$.
We can apply the functor $(-)^*$ of \autoref{Xstar} to these diagrams in $\sV_{\mI}$ to obtain commutative diagrams in $\LA[\sV_{\mI}]$.  These left structures, being given by covariant functors, do not fit into the monoidal $2$-sided bar construction  as recalled in the introduction, but they do fit into a  variant that we will describe in the next section. 

To fit $\sC$-algebras into the monoidal bar construction, we instead exploit $i_0 $ and $p_0$ of \autoref{notn3}.   Recall that \autoref {iodot} gives 
that $\sC\odot i_0X \iso i_0\bC X$ in $\LA^{op}_{\mI}[\sV]$.  Applying $i_0$ to the diagrams above and using this commutation relation together with the observation that $p_0\com i_0 = \id$, we obtain the following examples of left $\sC$-modules.

\begin{lem}\label{iodot2}  For a unital operad $\sC$ in $\sV$, application of $i_0$ and $p_0$ gives an isomorphism between the category of $\bC$-algebras in $\sV_{\mI}$ (equivalently $\sC$-algebras in $\sV_{\mI}$) and of left $\sC$-modules that are $\emptyset$ above level $0$. 
\end{lem}

\begin{rem} The books \cite{Fresse0, Fresse, Fresse1, LV} are in whole \cite{Fresse0} or in part devoted to the study of modules over operads in categories of $\SI$-sequences, with primary focus on algebraic contexts.  Much of their work can be adapted to give the analogous discussion of modules over unital operads in categories of $\LA$-sequences.  We will give one example of that  in \autoref{Nmods}.
\end{rem}

\section{A comparison of bar constructions}\label{bar}

Assume given a triple $(\sM,\sC, Y)$ consisting of a unital operad $\sC$, viewed as a monoid in 
$\LA_{\mI}^{op}[\sV]$, a right $\sC$-module $\sM$ in $\LA^{op}[\sV]_{\mI_0}$, and a left $\sC$-module $Y$ in 
$\LA^{op}[\sV]_{\mI_0}$.  The general monoidal bar construction described in the introduction gives a 
bar construction $B_{\bullet}(\sM, \sC, Y)$, which is a simplicial object in $\LA^{op}[\sV]_{\mI_0}$.

We specialize by taking $Y=i_0X$,  where $X$ is a $\sC$-algebra in 
$\sV_{\mI}$.  Here $i_0X$ is as defined in \autoref{notn3}.  \autoref{iodot} implies that $i_0X$  inherits a left $\sC$-module structure from the action of $\sC$ on $X$. 

\begin{defn}
The $q$-simplex object of 
$B_\bullet(\sM,\sC, i_0X)$ is 
\begin{equation}\label{zeroth}
B_q(\sM, \sC, i_0X) = \sM \odot \sC^{\odot  q}\odot i_0X\\
\end{equation}
The face maps 
$$d_i: B_q(\sM, \sC, i_0X) \rtarr B_{q-1}(\sM, \sC, i_0X), \ \ 0\leq i\leq q, $$ 
are induced in order by the right action $\rho\colon \sM\otimes \sC \rtarr \sM$ if $i=0$,  the product $\ga\colon \sC\odot \sC\rtarr \sC$ applied on the $i$th and $(i+1)$st  copies of $\sC$ for $1\leq i\leq q-1$, and $\tha\colon \sC\odot i_0X \rtarr i_0 X$ if $i=q$.  The degeneracy maps 
$$ s_i: B_q(\sM,\sC, i_0 X) \rtarr B_{q+1}(\sM, \sC, i_0X), \ \ 0\leq i\leq q,$$ 
insert the identity map $\mI_1 \rtarr \sC$ in the $i$th slot.    
\end{defn}

We next construct a different  two-sided bar construction
 $B_\bullet(\sM,\sC, X^{\otimes*})$, where $X$ is again a $\sC$-algebra in $\sV_{\mI}$ and relate it to the one above.  This bar construction is a simplicial object in  $\sV_{\mI}$, which is what we want in
applications.
\begin{defn}
   The $q$-simplex object of $B_\bullet(\sM,\sC, X^{\otimes *})$ is 
\begin{equation}\label{first}
B_q(\sM, \sC, X^{\otimes *}) = (\sM \odot \sC^{\odot  q})\otimes_{\Lambda} X^{\otimes *}\\
\end{equation}
Except for the last face operation $d_q$ , which is given by $\tha\colon \sC\otimes_{\LA} X^{\otimes *} \rtarr X$, the face and degeneracy operations are as defined on $B_q(\sM, \sC, i_0X)$. 
\end{defn}

Both of these are essentially equivalent to a certain monadic bar construction.  We recall from \cite{MayGeo} that monadic bar constructions are defined on triples  $(\bF,\bC,X)$, where $\bC$ is a monad in some category $\sW$, $X$ is a $\bC$-algebra in $\sW$, and $\bF\colon \sW \rtarr \sZ$ is a $\bC$-functor  in some category $\sZ$, possibly  $\sZ = \sW$.  Here $\bF$ comes with an action natural transformation $\rh\colon \bF\bC\rtarr \bF$ such that the following diagrams commute.
\[\xymatrix{
 \bF   \ar[r]^-{\bF\et} \ar[dr]_{=} & \bF\bC  \ar[d]^{\tha} \\
 & F}
 \ \ \ \ \ \ \
 \xymatrix{
 \bF\bC\bC  \ar[d]_{\rho} \ar[r]^-{\bF\mu}  & \bF\bC   \ar[d]^{\rho} \\
\bF \bC   \ar[r]_-{\rho}  &  F \\} \]
The bar construction is then a simplicial object in $\sZ$  with $q$th simplex object
\begin{equation}
B_q(\bF,\bC,X) = \bF\bC^q X.
\end{equation}
The face and degeneracy maps in $\sZ$ are defined just as for the monoidal bar construction.

We see using \autoref{equivcomp} that the right action of $\sC$ on $\sM$ induces a right action of the monad $\bC$ on the functor $\bM\colon \sV_{\mI}\rtarr \sV_{\mI}$, which we recall is given by  $\bM X = \sM\otimes_{\LA}X^{\otimes *}$.

\begin{prop}      There are natural identifications
$$ p_0B_{\bullet}(\sM,\sC,i_0X) \iso B_{\bullet} (\bM,\bC,X) \iso B_{\bullet}(\sM,\sC,X^{\otimes *})$$ 
of simplicial objects in $\sV_{\mI}$.
\end{prop}
\begin{proof}  For the first, apply $p_0$ to the identification
$$ B_{\bullet}(\sM,\sC,i_0X) \iso i_0B_{\bullet} (\bM,\bC,X)$$
given by  \autoref{iodot}. For the second,  observe that \autoref{first} can be written in terms of 
\autoref{DXdefn} as 
\begin{equation}\label{second}
B_q(\sM, \sC, X^{\otimes *}) = (\bM \odot \bC^{\odot  q})(X).\\
\end{equation}
Then inductive use of \autoref{equivcomp} gives a natural isomorphism 
\begin{equation}\label{third}
B_q(\sM, \sC, X^{\otimes *}) \iso  \bM\bC^qX,
\end{equation}
with an evident reinterpretation of the face and degeneracy maps.
\end{proof}

\begin{rem}\label{intex}  The interesting examples of $\bC$-functors $F$ in the classical examples are not of the form $\bM$ for right $\sC$-modules $\sM$,.
\end{rem}

\section{Endomorphism operads and homomorphism monoids}\label{VWcon}
Recall that if $\sV$ is closed with internal hom objects denoted $\ul{\sV}(-,-)$ and $X\in \sV$,  there is an endomorphism operad $\End(X)$ in $\sV$ with $k$th term  $\ul{\sV}(X^{\otimes k},X)$.  Its unit map is $\id\colon \mI\rtarr \ul{\sV}(X,X)$ and its structural maps $\ga$ are composites 
\begin{equation}\label{Endogamma}
\xymatrix{\ul{\sV}(X^{\otimes k}, X) \otimes \ul{\sV}(X^{\otimes j_1}, X)\otimes \cdots \otimes \ul{\sV}(X^{\otimes j_k}, X)\ar[d]^{\id \otimes\otimes^k}\\
\ul{\sV}(X^{\otimes k}, X) \otimes \ul{\sV}(X^{\otimes j}, X^{\otimes k}) \ar[d]^{\com}\\
\ul{\sV}(X^{\otimes j}, X)\\}
\end{equation}
of $\otimes$ and composition, where $j=j_1 + \cdots + j_k$.  

Using the natural isomorphisms
\begin{equation}\label{hom}
\ul{\sV}(\sC(k)\otimes X^{\otimes k},  X) \iso \ul{\sV}(\sC(k),\ul{\sV}(X^{\otimes k}, X)),
\end{equation}
an action of an operad $\sC$ on $X$ can be identified with a morphism $\sC \rtarr \End(X)$ of operads  in $\sV$.

We generalize this to categories $\sW$ enriched over $\sV$. 
In this section and the next, we let $(\sW,\otimes,\mI_{\sW})$ be a symmetric monoidal category enriched in a (not necessarily closed) symmetric monoidal category $(\sV, \otimes, \mI)$.   Note that we do not distinguish notationally between $\otimes = \otimes_{\sW}$ and $\otimes = \otimes_{\sV}$, relying on context to indicate which is meant.  We write $\ul{\sW}(A,B)\in \sV$ for the morphism objects of $\sW$.  

\begin{lem}\label{relend}
Let $Y\in \sW$.  There is an endomorphism operad $\End_{\sW}Y$ in $\sV$ with $k$th term $\ul{\sW}(Y^{\otimes k},Y)$.
\end{lem}
\begin{proof}  Replacing $X$ by $Y$ and  $\ul{\sV}$ by $\ul{\sW}$ in \autoref{Endogamma}  gives the structural maps.
\end{proof} 

Recall that $\sW$ is tensored over $\sV$ if there is a functor $\otimes\colon \sV\times \sW\rtarr \sW$ with an adjunction
$$  \sW(X\otimes Y, Z) \iso \sV(X, \ul{\sW}(Y, Z)), $$
where $X\in \sV$ and $Y,Z\in \sW$.   For an operad $\sC$ in $\sV$ and $Y\in \sW$, it follows that
\begin{equation}\label{hom2}
\sW(\sC(k) \otimes Y^{\otimes k}, Y) \iso \sV(\sC(k),\ul{\sW}(Y^{\otimes k}, Y))
\end{equation}
Here we define an action of $\sC$ on an object $Y\in \sW$ to be a set of maps in $\sW$
$$\sC(k) \otimes Y^{\otimes k} \rtarr Y$$
satisfying the same unity, associativity and equivariance conditions that hold when $\sV= \sW$.   Then \autoref{hom2}  gives the following conclusion.

\begin{prop}  If $\sW$ is enriched and tensored over $\sV$ and $\sC$ is an operad in $\sV$, then $\sC$-algebras $Y$ in $\sW$ correspond bijectively to  maps $\sC\rtarr \End_{\sW}(Y)$ of operads in $\sV$. 
\end{prop}

\begin{rem}  Good point-set level categories of spectra are enriched and tensored over the category of (unbased) spaces.  Intuitively, and rigorously in most models. such as those of \cite{LMS, EKMM, MM, MMSS}, for a space $X$ and spectrum $Y$, $X\otimes Y = \SI^{\infty}(X_+) \sma Y$.  Actions on spectra of operads of spaces have been used for decades, but the relationship with endomorphism operads seems to be new.
\end{rem}

However, in factorization homology, we will encounter a pair $(\sW,\sV)$ such that $\sW$ is enriched but not tensored over $\sV$.   We still have endomorphism operads $\End_{\sW}Y$, but we are interested in unital operads.  Since $Y^{\otimes 0} = \mI_{\sW}$, $\End_{\sW}(Y)$ is only unital when $\ul{\sW}(\mI_{\sW}, Y) = \mI$.  That leads us to the following construction.

\begin{con}
  \label{con:HW}
  Assume that $\sW$ is $\sV$-enriched symmetric monoidal. This means that
    the symmetric monoidal product $\otimes:\sW \times \sW \rtarr \sW$ induces a
    $\sV$-enriched functor $\sW \otimes \sW \to \sW$, where the tensor product of
    $\sV$-enriched categories is as defined in \cite[p12]{KellyBasic}. More
    precisely, there are (suitably coherent) natural maps
\begin{equation*}
\ul{\sW}(X, Y) \otimes \ul{\sW}(X', Y') \to \ul{\sW}(X \otimes X', Y \otimes Y'),
\end{equation*} 
where $\otimes = \otimes_{\sV}$ on the left and $\otimes = \otimes_{\sW}$ on the right.

Assume further that $\ul{\sW}(\mI_{\sW} , Y) \cong \mI$ for all $Y \in \sW$.  
We construct a  category $\Hom_{\sW}$ enriched in 
$\LA_{\mI}^{op}[\sV]$ such that the objects are those of $\sW$ and, for $X,Y \in\sW$, the enriched morphism object is the functor $\LA^{op}\rtarr \sV$ defined by
$$\ul{\Hom}_{\sW}(X,Y) = \ul{\sW} (X^{\otimes *}, Y).$$
Since $X^{\otimes 0} = \mI_{\sW} $, our assumption gives that
$$ \ul{\Hom}_{\sW}(X,Y)(0) = \ul{\sW}(\mI_{\sW} , Y) \cong \mI.$$ 
(Thus $\Hom_{\sW}(X,Y)=\LA_{\mI}^{op}[\sV](\mI_0, \ul{Hom}_{\sW}(X,Y))$ is a point for all $X$ and $Y$.)
This verifies that $\ul{\Hom}_\sW(X,Y)$ is an object of $\LA_{\mI}^{op}[\sV]$.  Observe that, as $\LA$-sequences,                    
$$  \ul{\Hom}_{\sW}(Y,Y) = \End_{\sW}(Y).$$
\end{con}

The assumption on $\sW$ may well seem strange.  Without it, we could work in the category $\LA^{op}[\sV]_{\mI_0}$ of based rather than just unital $\LA$-sequences. We shall see an important topological example where the assumption holds in \autoref{fact}.  With it, we now show that our unital endomorphism operads $\End_{\sW}Y$ can be viewed as monoids in a monoidal category  enriched in $\LA_{\mI}^{op}[\sV]$. 

\begin{prop}\label{HCat}
 $\ul{\Hom}_{\sW}$ admits a composition that gives $\sW$ the structure of a category enriched in $\LA_{\mI}^{op}[\sV]$.
\end{prop}
\begin{proof}
For an object $X\in \sW$, the identity map $\mI \rtarr  \ul{\sW}(X,X)$ of the enriched category $\sW$ induces an identity map $\id: \mI_1 \rtarr 
\ul{\Hom}_{\sW}(X,X) =  \ul{\sW}(X^{\otimes *}, X)$. 
Note that $\id$ is a morphism in $\LA^{op}_\mI[\sV]$, that is, a natural transformation of functors. 
Via the following composites, the product map relating $\otimes$ of $\sV$ to $\otimes$ of $\sW$ followed by composition
$$  \ul{\sW}(Y, Z) \otimes  \ul{\sW}(X, Y) \rtarr  \ul{\sW}(X, Z)$$
 induce the required composition morphisms in $\LA^{op}_\mI[\sV]$:
\begin{equation*}
  \begin{tikzcd}
\ul{\Hom}_{\sW}(Y,Z) \odot \ul{\Hom}_{\sW}(X,Y)  \ar[d, "="] \\
  \int^{\bk, \bj_1, \ldots, \bj_k} \ul{\sW}(Y^{\otimes k}, Z) \otimes \ul{\sW}(X^{\otimes j_1},
     Y) \otimes \cdots \otimes \ul{\sW}(X^{\otimes j_k}, Y)  \otimes \LA(-, j_1 + \ldots + j_k) \ar[d] \\
     \int^{\bk, \bj_1, \ldots, \bj_k} \ul{\sW}(Y^{\otimes k}, Z) \otimes  \ul{\sW}(X^{\otimes
        (j_1+\cdots+j_k)}, Y^{\otimes k}) \otimes \LA(-, j_1 + \ldots + j_k)  \ar[d] \\
      \int^{\bj_1, \ldots, \bj_k}    \ul{\sW}(X^{\otimes (j_1+\cdots+j_k)}, Z) \otimes \LA(-, j_1 + \ldots + j_k)\ar[d, "\iso"] \\
      \ul{\sW}(X^{\otimes *}, Z)\ar[d, "="] \\
 \ul{\Hom}_{\sW}(X,Z) 
  \end{tikzcd}
\end{equation*}
 The isomorphism is another application of \autoref{coYon}. The category axioms are verified by easy diagram chasing. 
\end{proof}

\begin{cor}\label{Hoperad}
Let $X$, $Y$, $Z$ be objects of $\Hom_\sW$.  Then $\ul{\Hom}_\sW(Y,Y)$ is a monoid in the enriching category $\LA^{op}_\mI[\sV]$ with right module $\ul{\Hom}_\sW(Y,Z)$ and left module $\ul{\Hom}_\sW(X,Y)$.  Thought of operadically, $\ul{\Hom}_{\sW}(Y,Y)$ is the unital endomorphism operad $\End_{\sW}(Y)$ in $\sV$.
\end{cor}
\begin{proof}
The product map and the left and right actions of the monoid are given by composition in $\ul{\Hom}_{\sW}$; the unit is given by 
$\id: \mI_1 \rtarr \ul{\Hom}_\sW(Y,Y)$. The required commutative diagrams are immediate from the category axioms.
\end{proof}

\section{Factorization homology and nonabelian Poincar\'e duality}\label{fact}

Here we shall give a quick sketch of the geometric definition of equivariant factorization homology and describe nonabelian Poincar\'e duality.   Full topological details are in the third author's thesis \cite{Foling}, but with less categorical contextualization.   

We fix a finite group $G$ and take $\sV$ in  \autoref{VWcon}  to be the category $G\sU$ of (unbased) $G$-spaces and $G$-maps.  It is enriched in itself, with $\ul{G\sU}(X,Y)$ being the $G$-space of maps $X\rtarr Y$, with $G$ acting by conjugation.  We shall shortly also use the category $G\sT$ of (based) $G$-spaces and based $G$-maps; $\ul{G\sT}(X,Y)$ is the based $G$-space of based maps $X\rtarr Y$, with $G$ acting by conjugation. 

We also  fix a finite dimensional $G$-representation $V$.   We shall generally omit $G$ and $V$ from the notations.  We  take $\sW$ in \autoref{VWcon}  to be the category $\sE{mb}^{\fr}$ of $V$-framed $G$-manifolds and $V$-framed embeddings.   Here, a $G$-manifold is $V$-framed if we are given an isomorphism $\mathrm{T}M \cong M \times V$ of equivariant vector bundles; an embedding $f: L \rtarr  M$ of $V$-framed $G$-manifolds is $V$-framed if we are given a path in the $V$-framing space of $L$ from its given $V$-framing to the one induced by $f$ and $M$. We then have the $G$-space 
$\sE{mb}^{\fr}(L,M)$ of $V$-framed embeddings, giving the enrichment of $\sE{mb}^{\fr}$ in $G\sU$.  Note, however, that $\sE{mb}^{\fr}$ is not tensored over $G\sU$ since a product of a $G$-space and a $G$-manifold is not a $G$-manifold.

We give $\sE{mb}^{\fr}$ a symmetric monoidal structure by taking  its product to be disjoint union (which
is not the categorical coproduct here) and taking its unit to be the empty manifold.  
With these definitions, we have placed ourselves in exactly the context discussed in \autoref{VWcon}.
Since there is a unique map from the empty manifold to any manifold, the assumption of \autoref{con:HW} is satisfied. Therefore 
\autoref{HCat} and \autoref{Hoperad} apply. 

The representation $V$, considered as a $G$-manifold, is canonically $V$-framed and $\sE{mb}^{\fr}(V,M) \simeq M$.  We rename the functor $\ul{\Hom}_{\sE{mb}^{\fr}}(V,M)$  of \autoref{con:HW}, calling it $\Mfr$. Explicitly,  for a $V$-framed manifold $M$, the $\LA$-sequence
$$\Mfr \colon \LA^{op} \rtarr G\sU$$
is given by
$$ \Mfr(\bk) = \sE{mb}^{\fr}(^kV, M). $$
Here $^kV$ denotes the disjoint union of $k$ copies of $V$; $\SI_k$ acts by permuting  these copies, 
and an injection $\si_i\colon \bk \rtarr \mathbf{k+1}$ induces the map obtained by
restricting embeddings of $k+1$ copies of $V$ to the $k$ copies obtained by omitting the $i$th one. 

We obtain a functor 
$$\bD_M^{\mathrm{fr}}\colon \ul{G\sT} \rtarr \ul{G\sT}$$ 
by applying \autoref{DXdefn} to the functor $\Mfr$.  Thus $\bD_M^{\mathrm{fr}}X = \Mfr \otimes_{\LA} X^{\otimes *}$.

By specialization of  \autoref{Hoperad}, ${\Vfr}$ is an operad in $G\sU$.  It is called the framed little $V$-discs operad. Its $k$-th $G$-space is  $\sE{mb}^{\fr}(^kV, V)$.   There is an inclusion of the classical little discs operad $\sD_V$ in  ${\Vfr}$, and it is a levelwise equivalence of $G$-operads \cite{Foling}. We have the monad $\Vbfr$ associated to $\Vfr$, and it acts (from the right) on the functor  $\bD_M^{\mathrm{fr}}$ for any $V$-framed manifold $M$.  It also acts on the $V$th suspension functor $\SI^V$.

The geometric definition of factorization homology reads as follows.  Here we do pass to geometric realization from our simplicial bar constructions.

\begin{defn}   The equivariant factorization homology $\int^{\fr}_M A$ of a $V$-framed $G$-manifold $M$ with coefficients in a ${\Vfr}$-algebra $A$ is defined to be the two-sided monadic bar construction
$$ \int^{\fr}_M A = B(\bD_M^{\mathrm{fr}}, \bD_V^{\mathrm{fr}}, A). $$
Using \autoref{bar}, we can reinterpret this bar construction as $B(\Mfr,\Vfr,A^*)$.  
\end{defn}  

Define $E_V A = B(\SI^V, \bD_V^{\mathrm{fr}}, A)$.  This is a version of the $V$-fold delooping functor defined nonequivariantly in \cite{MayGeo} and equivariantly in \cite{GM2}.  By a result called the recognition principle, when $A$ is $G$-connected and $\bR\subset V$, $A$ is weak $G$-equivalent  as a $\Vfr$-algebra  to
$\OM^V E_VA$. We summarize the argument in the language here to help motivate the equivariant Poincar\'e duality theorem.  There is a natural map of $\Vfr$-algebras $B(\bD_V^{\mathrm{fr}}, \bD_V^{\mathrm{fr}}, A) \rtarr A$ which is a $G$-homotopy equivalence, so we are focusing on the case $M=V$ of factorization homology.   

There is a natural map of $\Vfr$-algebras $\al\colon \bD_V^{\mathrm{fr}} A \rtarr \OM^V\SI^VA$.  Part of a result called the approximation theorem says that $\al$ is a weak $G$-equivalence when $A$ is $G$-connected.   It then induces a weak $G$-equivalence of $\Vfr$-algebras
\begin{equation}\label{alpha}
 \al_*\colon B(\bD_V^{\mathrm{fr}}, \bD_V^{\mathrm{fr}}, A) \rtarr B(\OM^V\SI^V, \bD_V^{\mathrm{fr}}, A).
 \end{equation}
For a simplicial based $G$-space $K_{\bullet}$ and a based $G$-space $X$, there is a composite of a  natural isomorphism and an evaluation map 
$$  X\sma |\ul{G\sT}(X,K_{\bullet})| \iso | X\sma  \ul{G\sT}(X,K_{\bullet})| \rtarr |K|. $$
Its adjoint is a natural $G$-map 
$$ \ze\colon |\ul{G\sT}(X,K_{\bullet})| \rtarr \ul{G\sT}(X, |K_{\bullet}|).$$
By  \cite{Haus} or \cite[pp 495-496]{CW}, there are explicit verifiable conditions under which $\ze$ is a weak $G$-equivalence.  Taking $X= S^V$ and 
$K$ to be our simplicial bar construction, $\ze$ gives a weak $G$-equivalence of $\Vfr$-algebras
\begin{equation}\label{zeta}
 \ze\colon B(\OM^V\SI^V, \bD_V^{\mathrm{fr}}, A) \rtarr \OM^V B(\SI^V, \bD_V^{\mathrm{fr}}, A) = \OM^V E_VA.
 \end{equation}
 
These weak $G$-equivalences encode the recognition principle, and we ask for their analogs for a general $V$-framed $G$-manifold $M$.  Let $M^+$ denote the one-point compactification of $M$, which is based at the point at infinity. Thus $V^+ = S^V$ and $M^+$ is the disjoint union of $M$ and a $G$-fixed basepoint if $M$ is compact.   Equivariant Poincar\'e duality is a weak $G$- equivalence between $ \int^{\fr}_M A$ and the function $G$-space $\ul{G\sT}(M^+,E_VA )$.   By definition,  $\OM^V\SI^V = \ul{G\sT}(S^V, S^V) $.   Replacing $S^V$ by $M^+$ and arguing as above, we obtain the following  analog of \autoref{zeta}. 
\begin{equation}\label{zeta2}
 \ze\colon B( \ul{G\sT}(M^+, \SI^V ), \bD_V^{\mathrm{fr}}, A) \rtarr  \ul{G\sT}(M^+,B(\SI^V,\bD_V^{\mathrm{fr}}, A)) = \ul{G\sT}(M^+,E_VA).
 \end{equation}
 We seek an analog of $\al$.  That will be given by a ``scanning map", which is a natural transformation 
 $$\PH\colon \bD_M^{\mathrm{fr}} \rtarr  \ul{G\sT}(M^+, \SI^V ).$$
 In contrast with  $\al$, the definition of $\PH$ is not at all obvious and is deferred to the next section.  It induces
 \begin{equation}\label{PHI*}
 \PH_*\colon  \int^{\fr}_M A = B(\bD_M^{\mathrm{fr}}, \bD_V^{\mathrm{fr}}, A) \rtarr  B(\ul{G\sT}(M^+, \SI^V ), \bD_V^{\mathrm{fr}}, A).
 \end{equation}
 
\begin{thm}[Equivariant nonabelian Poincar\'e duality]\label{nonAb}  If $A$ is $G$-connected, then $\ze$ and $\PH_*$ are natural weak equivalences,  hence so is their composite
$$\ze\com \PH_*\colon  \int^{\fr}_M A \rtarr  \ul{G\sT}(M^+,E_VA).$$
\end{thm}

The point is that it is very hard to compute invariants of function $G$-spaces,whereas bar constructions have proven to be very amenable to calculation nonequivariantly;  it is hoped and expected that they will eventually also be so equivariantly.

\section{The definition of the scanning map $\PH$}\label{SCAN}

The use of $\LA$-sequences clarifies and simplifies the definition of $\PH$.  For a $V$-framed $G$-manifold $M$,  we have the functor 
$$\FM \colon \LA^{op} \rtarr G\sT$$ 
given by
$$ \FM(\bk) = \ul{G\sT}(M^+, {^k}S^V). $$
Here $^kS^V$ denotes the wedge of $k$ copies of $S^V$ and is identified with $(^kV)^+$. The right actions of the groups $\SI_k$ give the symmetric sequence.  For an ordered injection  $\si_i\colon  \mathbf{k} \rtarr \mathbf{k+1}$, the action of 
$\si_i$ is induced by the map ${^{k+1}}S^V \rtarr {^k}S^V$ that sends the $i$th wedge summand to the basepoint.  In particular, if $M$ is a point,
then $M_+ = S^0$ and $\FM$ is just  $^{*}\!S^V$.   With this definition in place, we can define the scanning  map
$\PH\colon  {\Mbfr}X\rtarr \ul{G\sT}(M^+,\SI^VX)$ for a $G$-space $X$.

\begin{defn}\label{scan}
For a based $G$-space $X$, $\PH$ is the following composite:
$$   \xymatrix{  
\bD_M^{\mathrm{fr}} X \ar[d]^{=} \\
  \Mfr \otimes_{\LA}X^* \ar[d]^{\pi_M\otimes_{\LA} \id}  \\
     \ul{G\sT}(M^+, \,^*S^V) \otimes_{\LA}  X^* \ar[d] ^{\ph} \\
     \ul{G\sT}(M^+,\,  ^*S^V \otimes_{\LA} X^*)  \ar[d]^{\ul{G\sT}(\id, \rh)}   \\
     \ul{G\sT}(M^+, \SI^V X)\\}  $$
\end{defn}

We shall define $\pi_M$, $\rh$, and $\ph$ in the following three subsections and
we shall prove the following result in the fourth.  The first three include a
few details that are needed in the fourth.  We shall see that $\pi_M$ is a
version of the classical Pontryagin-Thom map,  while $\rh$ and $\ph$ are
specializations of elementary general formal constructions with
$\LA$-sequences.

\begin{thm}\label{Dfunctors} The monad $\Vbfr$ associated to $\Vfr$ acts (from the right)
    on all functors displayed in \autoref{scan}, and $\pi_M\otimes_{\LA}\id$, $\ph$, and $\ul{G\sT}^{op}(\rh, \id)$  are maps of $\bD_V^{\mathrm{fr}}$-functors. 
\end{thm}

The following consequence is not easily seen by direct inspection of definitions.

\begin{cor} When $M = V$, and thus $M^+ = S^V$,  the scanning map $\PH$ coincides with the map $\al$.
\end{cor}
\begin{proof}  This is immediate from the fact that $\bD_V^{\mathrm{fr}}$ is the free $\bD_V^{\mathrm{fr}}$-algebra functor and $\PH$ and $ \al$ are  maps of $\bD_M^{\mathrm{fr}}$-algebras that agree when composed with the unit $\et\colon X \rtarr \bD_V^{\mathrm{fr}}X$.  Both composites are easily checked to be the unit  $X\rtarr \OM^V\SI^V X$ of the adjunction $(\SI^V,\OM^V)$.
\end{proof}

This will complete the details needed to make sense of \autoref{nonAb}. The intuitive idea of its plausibility is
that a mild generalization of the approximation theorem applies locally, so that
equivariant Poincar\'e duality can be seen as a globalization of the
approximation theorem.  The manifold $M$ is the union of orbits $G\times_H U$ where
$U$ is an $H$-contractible Euclidean subspace.  The map $\PH$  for  $G\times_H U$ is
close enough to $\al$ that the method of proof of the approximation theorem
gives that it is a weak $G$-equivalence.  Then the standard result that a local
weak $G$-equivalence is a weak $G$-equivalence applies.  Detailed proofs are given in 
 \cite{HHKWZ} and \cite{Foling}.                                                                                                                                                       

\subsection{The Pontryagin-Thom map $\pi_M$}\label{Vframe}

We use the classical Pontryagin-Thom construction to construct a natural transformation 
$$\pi_M\colon \Mfr \rtarr \FM$$ 
of functors $\LA^{op} \rtarr G\sU$,
where we implicitly compose $\FM$ with the forgetful functor $G\sT\rtarr G\sU$.
Recall that the classical Pontryagin-Thom construction 
$$p \colon \sE{mb}(L,M) \rtarr \ul{G\sT}(M^+,L^+) = \ul{G\sT}^{op}(L^+,M^+)$$
sends an (open) embedding  $f\colon L\rtarr M$ of manifolds to the map $p(f)$, where $p(f)$ sends all points 
of $M$ not in the image of $f$, including the point at infinity of $M$, to the point at infinity of $L$ and sends $f(L)$ to $L$ via $f^{-1}$. 
For $V$-framed $G$-manifolds $L$ and $M$, we continue to denote the composite map
$$ \sE{mb}^{\fr}(L,M) \rtarr \sE{mb}(L,M) \rtarr \ul{G\sT}(M^+,L^+) =\ul{G\sT}^{op}(L^+,M^+)$$
by $p$, where the first map forgets the framing data (see \cite{Foling}) and the second is the classical Pontryagin-Thom construction defined above. Taking $L =\, ^{k}\!V$ as $k$ varies, inspection of definitions shows that $p$ is natural, and we denote $p$ by $\pi_M$. 

Again by inspection, we have the following commutative diagram
  \begin{equation}\label{pcomp}
  \xymatrix{  
  \sE{mb}^{\fr}(L,M) \times \sE{mb}^{\fr}(K,L) \ar[d]_{p \times p}  \ar[r]^-{\com} & \sE{mb}^{\fr}(K,M) \ar[d]^{p}  \\
 \ul{G\sT}^{op}(L^+,M^+) \times \ul{G\sT}^{op}(K^{+},L^{+}) \ar[r]_-{\com} &  \ul{G\sT}^{op}(K^+,M^+) \\ } 
 \end{equation}
 for $V$-framed manifolds $K$, $L$, and $M$.  The diagram explains why we shall often think about $p$ as taking values in the opposite category of $\ul{G\sT}$ from now on.  The following remark expands on this.
  
\begin{rem}\label{trans} In the language of \autoref{con:HW}, let $\sW' = (\ul{G\sT}^{op}, \vee, *)$, which is enriched in $G\sU$.  Then 
$$\FM = \ul{G\sT}^{op}({}^*S^V,M^+) = \ul{\Hom}_{\sW'}(S^V,M^+) \in \Lambda^{op}_{*}[G\sU].$$
Thus one-point compactification on objects and $\pi$ on morphisms  gives a functor $(-)^+:\sW \rtarr \sW'$ and
$\pi$ is a transformation $\ul{\Hom}_{\sW}(V,M) \rtarr \ul{\Hom}_{\sW'}(S^V,M^+)$.   Comparing composition  in $\sW$ and in $\sW'$, as defined in \autoref{HCat}, we obtain the following commutative diagram
\begin{equation}\label{WWprime}
\xymatrix{
\sE{mb}^{\fr}({}^*V,M) \odot \sE{mb}^{\fr}({}^*V,V)    \ar[r]^-{\com} \ar[d]_{\pi_M\odot \pi_V} & \sE{mb}^{\fr}({}^*V, M) \ar[d]^{\pi_M} \\
 \ul{G\sT}^{op}({}^*S^V,M^+) \odot  \ul{G\sT}^{op}({}^*S^V,S^V) \ar[r]_-{\com}  &   \ul{G\sT}^{op}({}^*S^V,M^+).\\}
 \end{equation}
\end{rem}

\subsection{The natural map $\rh$}

The maps $\rh$ and $\si$ are elementary and general, not specific to the case $\sV = G\sU$ of present interest.  We use topological notations, but the reader should keep the generality in mind, thinking of a general cartesian monoidal category $\sV$.  Given $X\in G\sT$, we have the usual functor  $X^*\colon \LA \rtarr G\sT$ given by the cartesian powers of $X$ with their permutations and with injections
sent to basepoint inclusions as in \autoref{Xstar}.  We also have a functor  $^{*\!}X\colon
\LA^{op}\rtarr \sV_*$ given by the ``copowers'' or wedges of copies of $X$ with their permutations of wedge summands and with injections sent to ``coprojections'' that send copies of $X$ not in the image of an injection to the basepoint, with the other summands sent by the identity to the corresponding wedge summands.   We then have the following observation.

\begin{lem}\label{note1}  For $X,Y\in G\sT$, there is a natural map 
$$\rh\colon ^{*}Y\otimes_{\LA} X^{*} \rtarr Y\sma X.$$
\end{lem}
\begin{proof}
We can identify $^k Y\times X^k$ with  the wedge of $k$-copies of  $Y\times X^k$.  These copies are ordered in the natural way, as are the coordinates of  $X^k$.   Define 
$$\rh_k\colon ^k Y\times X^k \rtarr Y\sma X$$ by letting its restriction  to the $i$th copy of $Y\times X^k$ be the composite  
$$ \xymatrix  {Y\times X^k  \ar[r]^-{\id\times \pi_i}  & Y\times  X \ar[r]^-{q} &  Y\sma X\\} $$
of the $i$-th projection and the quotient map.  
By transporting the right action of $\SI_k$ to the corresponding left action, we can view the action of $\SI_k$ on $^k Y\times X^k$ as diagonal, and then it is clear that $\rh_k$ is equivariant, where $\SI_k$ acts trivially on $Y\sma X$.  Consider the ordered injection $\si_i\colon \bk \rtarr \mathbf{k+1}$, $1\leq i\leq k+1$ that skips $i$ in the target.  We claim that the following diagram commutes.
$$\xymatrix{   ^{k+1}Y\times X^k  \ar[d]_{\si_i^{*}\times \id} \ar[r]^-{\id\times \si_i}  
& ^{k+1}Y\times X^{k+1} \ar[d]^{\rh_{k+1}} \\
^kY\times X^k \ar[r]_{\rh_k}   & Y\sma X\\}\\ $$
On the $j$th summand  $Y\times X^k$ for $j\neq i$, both composites are $q\com (\id\times \pi_j)$ for $j<i$ and for 
$j>i$, and both composites are the trivial map to the basepoint when $j = i$. Therefore the $\rh_k$ induce the desired map $\rh$ from 
$^{*}Y\otimes_{\LA} X^{*}$ to $Y\sma X$.
\end{proof}
 
\subsection{The natural map $\ph$}  In a general context, we assume that $\sV$ is closed, but we again focus on $\sV = G\sU$.
For $Z\in G\sU$ and  $\sD\colon \LA^{op}\rtarr G\sU$, define
$\ul{G\sU}(Z,\sD) \colon \LA^{op}\rtarr G\sU$  by 
$$  \ul{G\sU}(Z,\sD)(\bk) = \ul{G\sU}(Z,\sD(\bk)).$$
\begin{lem}\label{note2}  For $\sD\colon \LA^{op}\rtarr G\sU$ and $\sX\colon \LA\rtarr G\sU$, there is a natural map
$$\ph \colon \ul{G\sU}(Z,\sD)\otimes_{\LA} \sX \rtarr \ul{G\sU}(Z, \sD\otimes_{\LA} \sX)$$
 in $G\sU$.
\end{lem}
\begin{proof}
Recall that for $G$-spaces $X$, $Y$, $Z$ in $G\sU$ we have the usual natural map
$$   \ul{G\sU}(Z,X) \times Y  \rtarr \ul{G\sU}(Z, X\times Y)$$
adjoint to the evaluation map
$$   \xymatrix@1{ \ul{G\sU}(Z,X)\times Y \times Z \iso \ul{G\sU}(Z,X) \times Z \times Y   \ar[r]^-{\epz\times \id} & X\times Y. \\}  $$
This specializes to give maps 
$$  \ul{G\sU}(Z, \sD(\bk)) \times \sX(\bk)   \rtarr  \ul{G\sU}(Z,\sD(\bk) \otimes  \sX(\bk)). $$
These maps are easily seen to be $\SI_k$-equivariant, and a diagram analogous to the one in the proof of \autoref{note1} shows that they pass to coequalizers to give the desired map $\ph$. 
\end{proof}
We shall make use of two specializations of this general construction. The first gives the map $\ph$ in \autoref{scan}. 
\begin{lem}\label{note2a}   For based $G$-spaces $X,Y,Z$,  there is a natural map of based $G$-spaces
$$   \ph\colon  \ul{G\sU}(Z,\, ^*X)\otimes_{\LA} Y^* \rtarr \ul{G\sU}(Z,\, ^*X\otimes_{\LA} Y^*).$$
\end{lem}

\begin{lem}\label{note2b}
For $\LA$-sequences $\sD,\sE\in \LA^{op}[G\sU]_{\mI_0}$, there is a natural map of $\LA$-sequences
$$   \ph\colon  \ul{G\sU}(Z,\, \sD )\odot \sE\rtarr \ul{G\sU}(Z,\sD\odot \sE).$$
\end{lem}
\begin{proof}
We define $\ph$ by taking $\sX$ in \autoref{note2} to be $\sE^{\boxtimes *}(\bn)$ for varying $n$.
\end{proof}

\subsection{The proof of \autoref{Dfunctors}}
We interpret $\ul{G\sT}(M^+,-)$ as $\ul{G\sT}^{op}(-,M^+)$ in this proof. 
The claim is that if we replace $X$ by $\Vbfr X$ in \autoref{scan}, then the resulting composite maps to the composite as written, with unit and associativity conditions holding at each level.  We will leave verification of the unit and associativity conditions to the reader, focusing on constructing the three maps of $\Vbfr$-functors.   To abbreviate notation in the verifications sketched below, let us write

$$  \bT_M^a X = \ul{G\sT}^{op}({}^*S^V, M^+)\otimes_{\LA} X^*  $$

$$  \bT_M^b X = \ul{G\sT}^{op}({}^*S^V\otimes_{\LA} X^*, M^+)  $$

$$ \bT_M^c X = \ul{G\sT}^{op}(\SI^VX,M^+)  $$

\noindent
Then the claim is that we can construct horizontal action maps that make make the following diagram commute, where we abbreviate 
$\pi_M\otimes_{\LA}\id$ to $\pi_M$ and $\ul{G\sT}^{op}(\rh,\id)$ to $\rh_*$. 

\begin{equation}\label{scheme}
 \xymatrix{  
\Mbfr \Vbfr X  \ar[d]_{\pi_M} \ar[r]^-{\tha} & \Mbfr \ar[d]^{\pi_M} \\
\bT_M^a  \Vbfr X  \ar[d]_{\ph} \ar[r]^-{\tha^a} & \bT_M^a \ar[d]^{\phi}\\
\bT_M^b \Vbfr X  \ar[d]_{\rh_*} \ar[r]^-{\tha^b} & \bT_M^b  \ar[d]^{\rh_*}\\
\bT_M^c  \Vbfr X  \ar[r]_-{\tha^c} & \bT_M^c \\} 
\end{equation}

\noindent
By \autoref{equivcomp}, the top two composite functors in the left column are of the form $\bD\odot \bE$, as defined in \autoref{odot}, and the bottom two should be thought of as analogues.  For the top square,  we claim that the $\LA$-sequences $\Mfr$ and $\FM$ are right modules over the operad $\Vfr$, thought of as a monoid. This is shown by the following diagram, which is just a rewriting of \autoref{WWprime}.

{\small{

\begin{equation}\label{Yuck1}
  \xymatrix{  
\Mfr \odot \Vfr \ar[rr]^-{\com} \ar[d]_{\pi_M \odot \id} \ar[dr]^{\pi_M\odot\pi_V} &&  \Mfr \ar[d]^{\pi_M} \\
\ul{G\sT}^{op}({}^*S^V,M^+) \odot \Vfr  \ar[r]_-{\id\otimes\pi_V} &  \ul{G\sT}^{op}({}^*S^V,M^+) \odot  \ul{G\sT}^{op}({}^*S^V,S^V) \ar[r]_-{\com}  &   \ul{G\sT}^{op}({}^*S^V,M^+) \\}
  \end{equation}
  
 } } 
\noindent
Applying the functor $(-)\otimes_{\LA} X^{\otimes *}$, the horizontal arrows give $\tha$ and $\tha^a$ in \autoref{scheme} and the diagram gives the commutativity of its top square.

With $\sW'$ as defined in \autoref{trans}, we have
 \[   \ul{\Hom}_{\sW'}(S^V,\ast^+) =  {}^{*}\! S^V \ \ \text{and} \ \   \ul{\Hom}_{\sW'}(S^V,S^V) = \ul{G\sT}^{op}({}^{*}\! S^V, S^V). \]
Therefore  specialization of \autoref{HCat} to $\sW'$ gives that ${}^{*}\! S^V$ is a right $\Vfr$-module via the action
 
 \begin{equation}\label{suck1}
  \xymatrix@1{ {}^{*}\! S^V \odot \Vfr \ar[r]^-{\id\odot \pi_V} & {}^{*}\! S^V \odot \ul{G\sT}^{op}({}^*S^V,S^V) \ar[r]^-{\com} & {}^{*}\! S^V. \\}
  \end{equation}
 
\noindent
Applying $\ul{G\sT}^{op}(-,M^+)$ to \autoref{suck1} and then using \autoref{note2b} to compare with the bottom row of \autoref{Yuck1}, we obtain the following commutative diagram.
 
 {\small{

\begin{equation}\label{Yuck2}
  \xymatrix{  
\ul{G\sT}^{op}({}^*S^V,M^+) \odot \Vfr  \ar[r]^-{\id\otimes\pi_V} \ar[d]_{\ph} &  \ul{G\sT}^{op}({}^*S^V,M^+) \odot  \ul{G\sT}^{op}({}^*S^V,S^V) \ar[r]^-{\com} \ar[d]^{\ph} &   \ul{G\sT}^{op}({}^*S^V,M^+)  \ar[d]^{=} \\
\ul{G\sT}^{op}({}^*S^V \odot \Vfr,M^+) \ar[r]_-{(\id\odot\pi_V) _*} & \ul{G\sT}^{op}({}^{*}\! S^V \odot \ul{G\sT}^{op}({}^*S^V,S^V),M^+)
\ar[r]_-{\com_*} &  \ul{G\sT}^{op}({}^*S^V,M^+) \\}
 \end{equation}
  
 } } 
\noindent
Reinterpreting monadically by applying $(-)\otimes_{\LA}  X^{\otimes *}$, the bottom row gives $\tha^b$ in \autoref{scheme} and the diagram gives the commutativity of its middle square.

Analogously to \autoref{suck1}, but now starting on the functor level, we have the $\bD_V^{\mathrm{fr}}$-action on the functor 
$\Sigma^V$ given by
\begin{equation}\label{suck2}
  \Sigma^V \bD_V^{\mathrm{fr}}X \overset{\Sigma^V\Phi}{\rtarr }\Sigma^V \ul{G\sT}^{op}(\Sigma^V X,S^V) 
  \overset{\epsilon}{ \rtarr } \Sigma^V X.
\end{equation}

\noindent
Applying the functor $\ul{G\sT}^{op}(-,M^+)$, this gives the action $\tha^c$ in \autoref{scheme}.  Using 
$$\bT^a_VX = \ul{G\sT}^{op}({}^*S^V,S^V)\otimes_{\LA} X^*$$
and \autoref{note1}, we find that the following diagram commutes.

\begin{equation}\label{Yuck3}
\xymatrix{
{}^*S^V\otimes_{\LA} (\Vbfr X)^* \ar[r]^-{\id\otimes_{\LA} \pi_V} \ar[d]_{\rh} 
& {}^*S^V\otimes_{\LA}(\bT^a_V  X)^* \ar[r]^-{\com} \ar[d]^{\rh} &  {}^*S^V\otimes_{\LA} X^* \ar[d]^{\rh} \\
\SI^V\Vbfr X \ar[r]_-{\SI^V\pi_V} & \SI^V \bT^a_V X \ar[r]_-{\epz}  & \SI^V X \\}
\end{equation}
\noindent
Applying $\ul{G\sT}^{op}(-,M^+)$ to this diagram and chasing definitions, we find that it implies that the bottom square of \autoref{scheme} commutes.

\section{Remarks on the relevance of $\LA$-sequences to other topics}\label{remarks} 

While $\SI$-sequences (that is symmetric sequences or collections) are all over the operadic literature, $\LA$-sequences seldom appear.  However they are very often implicit, and extensions of results from $\SI$-sequences to $\LA$-sequences are very often of interest.  We give a  brief sampling here.

\subsection{Remarks on modules over the commutativity operad}\label{Nmods}
 Let $\sN$ denote the commutativity operad in $\sV$.  It has $\sN(k) = \mI$ for all $k\geq 0$. The symmetric group actions are trivial, and all of its structure maps are the identity.  We can regard $\sN$ as a monoid in the monoidal category $\SI^{op}[\sV]$ with its Kelly product $\odot_{\SI}$.  In \cite{Fresse0}, Fresse considered right $\sN$-modules as in \autoref{module}, but with $\LA$ replaced by $\SI$.  We prove the analog of his result in our context of $\LA$-sequences.
 
Let $\sF$ denote the category of finite based sets $\bn=\{0, 1, \cdots, n\}$
  with basepoint $0$, with morphisms all based maps.  Let $\sF_{>0}$ be the
  subcategory of $\sF$ with objects the $\bn$ and with those functions $\ph$
  such that $\ph^{-1}(0) = 0$; equivalently, it is the category of finite
  unbased sets.   Thus $\sF$ is the category of operators associated to $\sN$ (see \autoref{catopsC}), and $\sF_{>0}$ excludes, for example,  the projections $\de_j\colon \bn \rtarr \mathbf{1}$ that send $j$ to $1$ and $i$ to $0$ for $i\neq j$ that are crucial to the homotopical study of $\sF$-objects $X\colon \sF \rtarr \sV$. 

 In  \cite[5.1.6]{Fresse0}, Fresse proves that  right $\sN$-modules in $\SI^{op}[\sV]$ are exactly those $\SI$-sequences that extend to $\sF_{>0}^{op}[\sV]$.
Since $\SI\subset \LA \subset \sF_{>0}$, it is natural to expect the analog of Fresse's result to hold for $\LA$-sequences.

\begin{prop}\label{Fressetoo}
Right $\sN$-modules in $\LA^{op}[\sV]$ are those $\LA$-sequences that extend to $\aF_{>0} ^{op}[\sV]$.
\end{prop}
\begin{proof}
We first explain how the proof  works for $\SI$-sequences.  Thus let $\sE \in \SI^{op}[\sV]$ be a $\SI$-sequence. 
By \autoref{odotunits},
\begin{align*}
  (\sE \odot_{\SI} \comm) (\bn) & = (\sE \otimes_{\SI} \comm^{\boxtimes *} )(\bn)\\
    & =  \coprod_{k\ge 0}\coprod_{j_1+\cdots + j_k=n} \sE  (\bk) \otimes_{\Sigma_k}  \Big( (\comm
      (\bj_1) \otimes \ldots \otimes \comm (\bj_k)) \otimes_{\Sigma_{j_1} \times \ldots \times \Sigma_{j_k}} \Sigma_n \Big) \\
  &  =  \coprod_{k\ge 0}\coprod_{j_1+\cdots + j_k=n} \sE  (\bk) \otimes_{\Sigma_k}  \mI \otimes_{\Sigma_{j_1} \times \ldots \times \Sigma_{j_k}} \Sigma_n  \\
    &  = \coprod_{k\ge 0} \sE  (\bk) \otimes_{\Sigma_k}  \sF_{>0}( \{1,\cdots,n\}, \{1, \cdots, k\})
\end{align*}
Here, the last identity is because the index set $\{j_1+ \cdots + j_k = n\}$ is in bijection with the set of order preserving maps
$\ph\colon \{1,\cdots,n\} \to \{1, \cdots, k\}$ such that $j_i=|\ph^{-1}(i)|$. Applying $- \otimes_{\Sigma_{j_1} \times \ldots \times \Sigma_{j_k}} \Sigma_n$ then gives all maps $\bn \rtarr \bk$ in $\sF_{>0}$ as our indexing set.   The right $\comm$-module structure map
$$(\sE \odot_{\SI} \comm)(\bn) \to \sE(\bn)$$
then induces
\begin{equation*}
\ph^{*}: \sE(\bk) \to \sE(\bn).
\end{equation*}
See \cite[5.1.6]{Fresse0} for details, with some of the details credited to \cite{KapMan}.

Working instead with $\LA$-sequences, so taking $\sE \in \LA^{op}[\sV]$, we are given the extension from $\SI$ to $\LA$ and claim a further extension from $\LA$ to $\sF_{>0}$. Here
\begin{align*}
  (\sE \odot_{\LA} \comm) (\bn) & = (\sE \otimes_{\LA} \comm^{\boxtimes *} )(\bn) \\
  & = (\sE \otimes_{\SI} \comm^{\boxtimes *} )(\bn)/\sim \\
 &  =  \Bigg(\coprod_{k\ge 0}\coprod_{j_1+\cdots + j_k=n} \sE  (\bk) \otimes_{\Sigma_k}  \mI \otimes_{\Sigma_{j_1} \times
   \ldots \times \Sigma_{j_k}} \Sigma_n \Bigg)/\sim  \\
  &  =  \Bigg(\coprod_{k\ge 0} \sE  (\bk) \otimes_{\Sigma_k}  \sF_{>0}( \{1,\cdots,n\}, \{1, \cdots, k\})
      \Bigg)/ \sim 
\end{align*}
The equivalence relation $\sim$ is dictated by tensoring over $\LA$ rather than $\SI$. Its effect is to force the extension from $\SI$ to $\sF_{>0}$ defined above on the underlying $\SI$-sequence of the $\LA$-sequence $\sE$ to agree with the extension from  $\SI$ to $\LA$ given by $\sE$.  In effect, $\sim$ coequalizes  maps induced by the injections $\si_i\colon \mathbf{k} \rtarr \mathbf{k+1}$ with
 $1\leq i\leq k+1$.   For a map $\ps \colon \{1,\cdots,n\} \to \{1,\cdots,k+1\}$ in $\sF_{>0}$
 such that $\ps^{-1}(i) = \emptyset$ , there is a unique $\ph: \{1,\cdots,n\} \to
   \{1,\cdots,k\}$ such that  $\ps = \si_i\circ \ph$, and the term of 
$$  \sE(\mathbf{k+1}) \otimes \mI \otimes_{\Sigma_{j_1} \times \cdots \times \SI_0\times \cdots \times \Sigma_{j_k}} \Sigma_n$$ 
indexed on $\ps$ is attached by $\si_i^*\otimes \id$ to the term of
$$\sE (\bk)\otimes\mI \otimes_{\Sigma_{j_1} \times \ldots \times \Sigma_{j_k}} \Sigma_n $$
indexed on $\ph$.
\end{proof}

\subsection{A remark on model categories}
As pointed out in \cite[p.812]{BMAx}, it is hard to obtain a  model structure on the category of all operads in a model category $\sV$.  This is true since for any operad $\sC$ and $\sC$-algebra $A$, one can construct a new operad, for example the endomorphism operad $\End(A)$ when $\sV$ is closed, with $0$th term $A$, so that the category of operads  implicitly contains the category of algebras over any operad.   The situation is different for unital operads, where there are usually no unital endomorphism operads.  Starting with a cofibrantly generated monoidal model category $\mathcal{\sV}$, we have the projective model structure on the category $\Sigma^{op}[\mathcal{\sV}]$ of symmetric sequences in $\sV$. Its weak equivalences and fibrations are the levelwise weak equivalences and fibrations.  In Berger--Moerdijk \cite[Theorem 3.1]{BMAx}, mild conditions on $\mathcal{\sV}$ are given so that the category of unital operads\footnote{Unital operads are called reduced operads in \cite{BMAx}.} in 
$\mathcal{\sV}$ admits a cofibrantly generated model structure with weak equivalences and fibrations created by those of  $\Sigma^{op}[\mathcal{\sV}]$. Putting $\Lambda$-sequences  into the picture, the same argument works with $\SI$ replaced by $\LA$. Then the Quillen adjunction of \cite{BMAx} can be checked to be the composite of Quillen adjunctions
\begin{equation*}
  \begin{tikzcd}
    \Sigma^{op}[\mathcal{\sV}]^{\mathrm{proj}}  \ar[r, shift left ]  &
    \Lambda^{op}[\mathcal{\sV}]^{\mathrm{proj}} \ar[l, shift left, "\mathrm{forget}" ]  \ar[r,
    shift left ]  & \text{unital operads in } \mathcal{\sV}. \ar[l, shift left,  "\mathrm{forget}" ]
  \end{tikzcd}
\end{equation*}

\subsection{A remark on a generalization of the Kelly product}
The proof that the Kelly product is associative requires our assumption that $X\otimes -$ preserves finite colimits.  Motivated by coalgebraic structures, 
Ching in \cite[Definitions 2.10-2.12]{ChingComposition} generalizes the Kelly product to situations where this fails by defining functors 
  $$\mu_n : (\Sigma^{op}[\sV])^n \rtarr \Sigma^{op}[\sV]$$
   so that they equip $\Sigma^{op}[\sV]$ with a so-called normal oplax monoidal structure. 
  Using this, he can still make sense of monoids and bar constructions in
  $\Sigma^{op}[\sV]$, particularly for applications to cooperads.
  It would be interesting to do the same thing for $\Lambda^{op}[\sV]$, but due to
  the base point identifications, one would need
  a more complicated construction than Ching's to define the $\mu_n$. We have not pursued this.

\subsection{Remarks on dendroidal sets}

It is interesting to compare our context to dendroidal sets.
  A dendroidal set is a contravariant functor from
$\Omega$, the category of trees, to Sets (see \cite[3.4]{HeutsMoerdijk}). Roughly speaking, following \cite[p. xi]{HeutsMoerdijk}, the theory of dendroidal sets extends the theory of simplicial sets to allow “multiple inputs", just as the theory of multicategories, alias colored operads, extends the theory of categories.

The definition of $\Omega$ is a bit involved, and we only point out some  illustrative
features. Among the objects of $\Omega$, there are the ``$n$-corolla'' $C_n$, $n\geq 0$,  and their closures
$\overline{C_n}$ as displayed in Figure 1 below. Note that $\overline{C_0}=C_0$.
\begin{figure}[h!]\label{trees}
  \begin{subfigure}{0.18\textwidth}
    \begin{tikzpicture}[caption = {$C_0$}]
    \node {} [grow'=up]
    child {[fill] circle (2pt)};
    \end{tikzpicture}
  \end{subfigure}
\hfill
   \begin{subfigure}{0.18\textwidth}
    \begin{tikzpicture}[caption = $C_1$]
  \node {} [grow'=up]
  child {[fill] circle (2pt)
     child };
 \end{tikzpicture}
\end{subfigure}
\hfill
  \begin{subfigure}{0.18\textwidth}
    \begin{tikzpicture}[caption = {$C_2$}]
    \node {} [grow'=up]
    child {[fill] circle (2pt)
      child
      child
    };
    \end{tikzpicture}
  \end{subfigure}
 \hfill
  \begin{subfigure}{0.18\textwidth}
    \begin{tikzpicture}[caption = {$C_3$}]
    \node {} [grow'=up]
    child {[fill] circle (2pt)
      child
      child
      child
    };
    \end{tikzpicture}
  \end{subfigure}
  
  \begin{subfigure}{0.18\textwidth}
    \begin{tikzpicture}[caption = {$\overline{C_0}$}]
    \node {} [grow'=up]
    child {[fill] circle (2pt)};
    \end{tikzpicture}
  \end{subfigure}
  \hfill
   \begin{subfigure}{0.18\textwidth}
    \begin{tikzpicture}[caption = $\overline{C_1}$]
  \node {} [grow'=up]
  child {[fill] circle (2pt)
     child {[fill] circle (2pt)}};
 \end{tikzpicture}
\end{subfigure}
\hfill
  \begin{subfigure}{0.18\textwidth}
    \begin{tikzpicture}[caption = {$\overline{C_2}$}]
    \node {} [grow'=up]
    child {[fill] circle (2pt)
      child {[fill] circle (2pt)}
      child {[fill] circle (2pt)}
    };
    \end{tikzpicture}
  \end{subfigure}
 \hfill
  \begin{subfigure}{0.18\textwidth}
    \begin{tikzpicture}[caption = {$\overline{C_3}$}]
    \node {} [grow'=up]
    child {[fill] circle (2pt)
      child {[fill] circle (2pt)}
      child {[fill] circle (2pt)}
      child {[fill] circle (2pt)}
    };
    \end{tikzpicture}
  \end{subfigure}
  \label{fig:trees}
  \caption{Some trees}
\end{figure}

Morphisms are less straightforward.  In particular, there is no
morphism between $C_i$ and $C_j$ for $i \neq j$ (see \cite[top of p.97]{HeutsMoerdijk}); but there are morphisms between
$\overline{C_i}$ and $\overline{C_j}$ for each inclusion $\bi \to \bj$ (see \cite[Proposition 3.5]{HeutsMoerdijk}).

To compare with our context, the category $\Sigma$ is a full subcategory of $\Omega$ by
sending $\bn$ to $C_n$ for $n \geq  0$;  $\Lambda$ is a full subcategory by sending $\bn$ to
$\overline{C_n}$.  For a dendroidal set $\sX$, write $\sX_T$ for the value of 
$\sX$ on a tree $T$.  We have a forgetful functor from dendroidal sets to $\SI$-sequences that sends $\sX$ to 
$\{\sX_{C_n}\}$ and a  forgetful functor to $\LA$-sequences that sends $\sX$ to $\{\sX_{\overline{C_n}}\}$.  
In $\Omega$, there is a canonical morphism $T \to \overline{T}$ from a tree $T$ to its closure. By contravariance, these maps induce maps 
$\rh \colon \sX_{\overline{T}} \rtarr \sX_T$; $\sX$ is said to be closed  if $\rh$ is a bijection for all $T$.  

As shown in \cite[p. 108]{HeutsMoerdijk}, there is a functor $\mathrm{N}$ from operads in sets to dendroidal sets such that 
$\mathrm{N}(\sC)_{C_n} = \sC(n)$ for $\ n \geq 0$.  Putting things together, we have the following diagram.  The unlabelled solid arrows are forgetful functors; $\rho$ induces the natural transformation in the triangle on the right, and the outer triangle commutes.

\begin{equation*}
\begin{tikzcd}
  \text{Operads} \ar[rr," \mathrm{N}"] \ar[ddr]
  \ar[dr, dotted] & & \text{Dendroidal Sets} \ar[dl] \ar[ddl, ""{name = U,below}] \\
   & \Lambda \text{-sequences } \ar[d] \ar[to=U, Rightarrow] & \\
   & \Sigma \text{-sequences } &
\end{tikzcd}
\end{equation*}

For operads $\sC$ in sets, which is cartesian monoidal, unital means $\sC(0) = \ast$, and we have the following conclusion.

\begin{prop} For an operad $\sC$ in sets, $\mathrm{N}(\sC)$ is closed if and only if $\sC$ is unital, and these are precisely those 
operads $\sC$ for which  the dotted arrow in the left triangle exists making the diagram commute.
\end{prop}
\begin{proof}   For ``only if'', unraveling definitions in \cite[p108(e)]{HeutsMoerdijk} gives that  
$$\mathrm{N}(\sC)_{\overline{C_n}}=\sC(n) \times \sC(0)^n$$
 for $n > 0$.  For ``if", see \cite[p117]{HeutsMoerdijk}.
\end{proof}

\section{Appendix: operads as symmetric monoidal categories}\label{SMC}

This appendix is peripheral to our preferred take on operads, but it gives an illuminating alternative perspective.  It shows how operads in $\sV$ are equivalent to particular symmetric monoidal categories enriched in $\sV$ and it relates this perspective to the categories of operators that were introduced in \cite{MT}.   Due to \autoref{loxlox}, everything here works equally well for $\LA$-sequences as for $\SI$-sequences, but we shall focus on the latter throughout.

\subsection{The permutative envelope of operads}\label{SME}

The symmetric envelope of operads was first introduced in the $\infty$-category context
  in \cite[Definition 1.6.1]{Lurie} in greater generality than we will consider here.
For us,  it is another name for a construction in \cite[Proposition 3.1]{BBPTY2}, with the name coming from
such sources as \cite[p. 291]{CGw} or \cite[Section 3]{Lawson}.   Recall that a permutative category is a symmetric strict monoidal category, so that the product is strictly associative and unital.\footnote{The paper \cite{BBPTY2} uses the term strict symmetric monoidal category, without defining strict. According to Maria Basterra, symmetric strict, alias permutative, is what is meant.}   We prefer the term permutative envelope since that more accurately describes the construction in \cite{BBPTY2}.

The Day and Kelly products are implicitly used in the construction.  Making this use explicit gives a  conceptual reformulation of the construction. 
Using the universal property of the Day convolution  $\boxtimes = \boxtimes_{\SI}$, we obtain an associative and symmetric pairing 
\begin{equation}\label{pairing} 
\boxtimes\colon  \sD^{\boxtimes p}(\bm) \otimes \sE^{\boxtimes q}(\bn)  
\rtarr (\sD^{\boxtimes p} \boxtimes \sE^{\boxtimes q})(\bm + \bn) 
\end{equation}
between symmetric sequences in $\sV$.  Using \autoref{loxlox}, we see that it is induced in the evident way by $\otimes$ (in $\sV$) and the inclusion of $\SI_m\times \SI_n$ in $\SI_{m+n}$.  We are using $\boxtimes_{\SI}$ in this section, but 
\autoref{loxlox} shows that when $\sD$ and $\sE$ are $\LA$-sequences, we obtain the same symmetric sequence values if we replace 
$\boxtimes_{\SI}$ with $\boxtimes_{\LA}$. Note that $\sD = \mI_0$ is a unit for this pairing.  

Specializing the definition of $\odot = \odot_{\SI}$ gives $\sD^{\boxtimes p}\otimes_{\SI} \big (\sE^{\boxtimes *}(\bm)\big )  = (\sD^{\boxtimes p}\odot \sE)(\bm)$.  Therefore we have a natural map
\begin{equation}\label{compair}
\al\colon \sD^{\boxtimes p}(\bn) \otimes \sE^{\boxtimes \bn}(\bm) \rtarr  (\sD^{\boxtimes p}\odot \sE)(\bm) 
\iso (\sD\odot\sE)^{\boxtimes p}(\bm),
\end{equation}
for each $n$ and $m$, where the isomorphism is given by the second statement of \autoref{key}.
We focus on the case $\sD= \sE$ of these pairings.

We have the evident adjunction between functor categories
$${\bf{Fun}}(\SI, {\bf{Fun}}(\SI^{op}, \sV )) \iso  {\bf{Fun}}(\SI^{op}\times \SI, \sV)$$
(and similarly for $\LA$).  Thus, for a $\SI$-sequence $\sE$, we may view
$\sE^{\boxtimes *}$ as the adjoint bifunctor $\overline{\sE}\colon \SI^{op} \times \SI\rtarr \sV$ given by
\begin{equation}\label{tildeE}
\overline{\sE}(\bm, \bn) = \sE^{\boxtimes n}(\bm).
\end{equation}
This is $\mI$ if $m=n=0$ and is $\emptyset$ if $n=0$ and $m>0$.  
Of course, 
$$\overline{\sE}(\bm, \mathbf{1}) = \sE(m)$$
as a right $\SI_m$-space for all $m\geq 0$.  If $\sE$ is unital, then
$\overline{\sE}(\mathbf{0},\mathbf{m}) =\mI$ for all $m\geq 0$.

With $\sD=\sE$, the pairing \autoref{pairing} can be written as 
\begin{equation}\label{pairing2}
\boxtimes\colon  \overline{\sE}(\bm,\bp) \otimes \overline{\sE}(\bn,\bq) 
\rtarr  \overline{\sE}(\bm + \bn, \bp + \bq)
\end{equation}
and the map \autoref{compair} can be written as
\begin{equation}\label{compair2}
\al\colon \overline{\sE}(\bn,\bp) \otimes \overline{\sE}(\bm,\bn) \rtarr (\overline{\sE\odot\sE})(\bm,\bp).
\end{equation}
If $\sE$ is an operad, we have the operadic product
\begin{equation}\label{gammak}
\ga^{\boxtimes p} \colon (\sE\odot\sE)^{\boxtimes p}(\bm) \rtarr \sE^{\boxtimes p}(\bm).
\end{equation}
We rewrite \autoref{gammak} in adjoint notation as 
\begin{equation}\label{gammak2}
\overline{\ga} \colon (\overline{\sE\odot\sE})(\bm,\bp) \rtarr \overline{\sE}(\bm,\bp),
\end{equation}
Composing  \autoref{gammak2} with \autoref{compair2} (and now thinking of our overlines as underlines),
 we obtain a composition making $\overline{\sE}$ a category enriched in $\sV$.
 The pairing \autoref{pairing2} provides  $\overline{\sE}$ with a  $\sV$-enriched permutative structure.  The product is strictly associative and unital by the associativity and unity conditions in the definition of an operad.  A straightforward check of definitions gives the following result.  
\begin{prop}  If $\sC$ is an operad in $\sV$, then $\overline{\sC}$ is a permutative category with composition $\overline{\ga}\com \al$
as displayed in \autoref{compair2} and  \autoref{gammak2}.  Its object monoid is $\mathbb N$ and its product is 
$\boxtimes$, as displayed in \autoref{pairing2}.    
\end{prop}

\begin{defn}  We call the permutative  category $\overline{\sC}$ the permutative envelope of the operad $\sC$.
\end{defn}

We single out two observations about $\overline{\sC}$.    First, the unit $\mI\rtarr \sC(1)$ of $\sC$ induces a map of monoids  
$$i_n\colon \SI_n \rtarr \overline{\sC}(\bn,\bn),$$
where the target is a monoid under composition; explicitly, $i_n$ is the composite
$$\SI_n \iso \mI^{\otimes n} \otimes \SI_n \rtarr \sC(1)^{\otimes n}\otimes \SI_n \subset  \sC^{\boxtimes n}(\bn) = \sC(\bn,\bn).$$
The $i_n$ specify a functor  $i_* \colon \SI \rtarr \overline{\sC}$.

Second, remembering that $\sC(n) = \overline{\sC}(\bn,\mathbf{1})$,  the composite

\begin{equation}\label{eq:permute} 
\xymatrix{ \coprod_{m_1 + \cdots + m_n = m}  \big(\overline{\sC}({\bm_1},\mathbf{1}) \otimes \cdots \otimes \overline{\sC}({\bm_n},\mathbf{1}) \big)
\otimes_{\SI_{m_1}\times \cdots \times \SI_{m_n}} \SI_m  \ar[d]^{\boxtimes\otimes i_m} \\
\overline{\sC}(\bm,\bn) \otimes_{\SI_{m_1}\times \cdots \times \SI_{m_n}} \overline{\sC}(\bm,\bm) \ar[d]^{\com}\\
\overline{\sC}(\bm,\bn)\\}  
\end{equation} 
is an isomorphism of right $\SI_m$ and left  $\SI_n$ objects in $\sV$ for
each $m$ and $n$.  Here $\ta\in\SI_n$ acts from the left on the source by permuting the $n$ tensor factors as $\ta$ permutes $n$ letters.

The case $n=1$ is worth emphasizing.  In the definition of an operad, the structure map $\gamma$ is required to be equivariant with respect to block permutations, whereas in contrast, $\circ: \overline{\sC}(\bn, \mathbf{1}) \otimes \overline{\sC}(\bm,\bn)
  \rtarr \overline{\sC}(\bm, \mathbf{1})$ is equivariant with respect to the  full $\Sigma_m$ and is also $\SI_n$-equivariant, where $\SI_n$ acts trivially on 
  $\overline{\sC}(\bm,\mathbf{1})$. This is so because \autoref{eq:permute} has identified
  $\overline{\sC}(\bm,\bn)$ with a symmetrized version of the input $\sC(m_1)\otimes \cdots \otimes \sC(m_n)$. 

Conversely, suppose that we have a permutative category $(\sK,\boxtimes)$ with object monoid $\mathbb N$ together with functors
$i_*\colon \SI \rtarr \sK$ sending $\SI_n$ to $ \rtarr \sK(\bn,\bn)$ such that the composite 
\begin{equation}
\label{eq:symmetrize}
\xymatrix{ \coprod_{m_1 + \cdots + m_n = m}  \big(\sK({\bm_1},\mathbf{1}) \otimes \cdots \otimes \sK({\bm_n},\mathbf{1}) \big) \otimes_{\SI_{_1}\times \cdots \times \SI_{m_n}} \SI_m  \ar[d]^{\boxtimes\otimes i_m} \\
\sK(\bm,\bn) \otimes_{\SI_{m_1}\times \cdots \times \SI_{m_n}} \sK(\bm,\bm) \ar[d]^{\com}\\
\sK(\bm,\bn)\\}  
\end{equation}
is an isomorphism of right $\SI_m$ and left  $\SI_n$ objects in $\sV$ for each $m\geq 1$ and $n\geq 0$. Then the right $\SI_n$-spaces $\sC(n) = \sK(\bn,\mathbf{1})$ specify an underlying operad with structure maps $\ga$ given by the composites
\[ \xymatrix@1{ \gamma:  \sC(n)\otimes \sC(m_1)\otimes \cdots \otimes \sC(m_n) \ar[r]^-{\id\otimes \boxtimes} & 
\sK(\bn,\mathbf{1})\otimes \sK(\bm,\bn) \ar[r]^-{\com}  & \sC(m), \\} \]
where $m= m_1+ \cdots + m_k$.  Another direct verification shows that $\sK$ is then isomorphic to the permutative envelope 
$\overline{\sC}$ as a permutative category.
Similarly, if we start with an
operad $\sC$, then $\sC$ is isomorphic as an operad to the underlying operad of
$\overline{\sC}$.   Summarizing and elaborating, we have the following version of
\cite[Proposition 3.1]{BBPTY2}.
\begin{thm}  The permutative  envelope and the underlying operad functors specify an
 equivalence of categories between the  category of operads in $\sV$ and a full subcategory of the
  category of $\sV$-enriched permutative categories with object monoid $\bN$ and permutative functors between them.
\end{thm}

\subsection{The relationship with categories of operators}\label{COO}

We consider two constructions on symmetric sequences $\sE$ in $\sV$.  Let $\sF$ and $\sF_{>0}$ be as in \autoref{Nmods}. 
Clearly $\SI\subset \LA
\subset \sF_{>0} \subset \sF$; $\LA$ and $\sF_{>0}$ are isomorphic to the corresponding categories of unbased finite sets.
Define
\[  \overline{\sE}(\bm,\bn) = \coprod_{\ph\in \sF_{>0}(\bm,\bn)}\,  \bigotimes_{1\leq j\leq n} \sE(|\ph^{-1}(j)|) \]
\[  \widehat{\sE}(\bm,\bn) = \coprod_{\ph\in \sF(\bm,\bn)}\,  \bigotimes_{1\leq j\leq n} \sE(|\ph^{-1}(j)|)  \]
Using composition in $\sF$, these are right $\SI_m$ and left $\SI_n$ objects of $\sV$, where the left action uses permutations of tensor factors. 

We show that this new definiton of $\overline{\sE}$ is isomorphic to $\sE^{\boxtimes n}(\bm)$,  so that it agrees with the old definition in \autoref{tildeE}.

 \begin{lem}\label{symseqiso} For symmetric sequences $\sE$, there is a natural isomorphism
  $$\om\colon \sE^{\boxtimes n}(\bm)  \rtarr \overline{\sE}(\bm,\bn)$$
 of  right $\SI_m$ and left $\SI_n$ objects of $\sV$.
 \end{lem}
 \begin{proof}
 Recall the description of the domain from the first statement of \autoref{loxlox}. 
 Since we require $\om$ to be a map of right $\SI_n$-objects, it suffices to define it on each component 
 \begin{equation}\label{component}
 \sE(m_1) \otimes \cdots \otimes \sE(m_n)
 \end{equation}
of the component
\begin{equation}\label{componenttoo}
 \coprod_{m_1 + \cdots + m_n = m}  \sE(m_1) \otimes \cdots \otimes \sE(m_n)
 \end{equation}
of the identity element of $\SI_m$ in $\sE^{\boxtimes n}(\bm)$. 
We use the given partition of $m$ into $n$ numbers to define a $\phi \in \sF_{>0}(\bm,\bn)$ 
for this component.   Explicitly, let $m_{j,i} = m_1 +\cdots + m_{j-1} +i$ for 
$1\leq j\leq n$ and $1\leq i\leq m_j$. Of course, if $m_j = 0$, there are  no such $i$. This
specifies an identification of $\bm$ in blocks, and we define $\ph\colon \bm\rtarr \bn$ by $\ph(0) =
0$ and $\ph(m_{j,i}) = j$; we later denote this $\ph$ by $\ph_{m_1,\cdots,m_n}$.  Clearly
$\ph^{-1}(j) = m_j$ for $1\leq j\leq n$.   Via the identity map on \autoref{component}  and passage
to coproducts, this defines $\om$ on the identity component \autoref{componenttoo}. It is clear that $\om$ defined by right $\SI_m$-equivariance is then  a left $\SI_n$-map.  To see that $\om$ is an isomorphism, let $\ph\colon \bm\rtarr \bn$ be any map in $\sF_{>0}$ and let $m_j = |\ph^{-1}(j)|$ for $1\leq j\leq n$.  There is a $\si\in \SI_m$ such that $\ph_{m_1,\cdots,m_n} \com \si = \ph$; for example, we can require $\si$ to map
the elements of $\ph^{-1}(j)$, in order, to the elements ${m_{j,i}}$, in order.  Then $\ph_{m_1,\cdots,m_n} \com \ta\com  \si = \ph$ 
for any $\ta\in \SI_{m_1}\times \cdots \times \SI_{m_n}$.  This implies that $\om$ maps components bijectively and is an  isomorphism
on each component.
\end{proof}

Now let $\sC$ be an operad.  Via the isomorphisms of \autoref{symseqiso}, we may give $\sC$ a structure of permutative category enriched in 
$\sV$ and identify it with the permutative envelope of $\sC$.  We relate it to the operadic  examples of  categories of operators, as defined in \cite{MT}  or, equivariantly,  \cite[Section 5.3]{MMO}; the latter allows an equivariant generalization of the following results.

\begin{defn}\label{catopsC}  The category of operators $\widehat{\sC}$ of an
  operad $\sC$ is the permutative category $\widehat{\sC}$ enriched in
  $\sV$ with objects $\bn$ and hom objects $\widehat{\sC}(\bm,\bn)$.  Its product 
  $$\otimes\colon \widehat{\sC}(\bm, \bn)\times \widehat{\sC}(\bp,\bq)\rtarr \widehat{\sC}(\mathbf{m+p},\mathbf{n+q})$$  
  is defined using the evident wedge product  $\wed\colon \sF\times\sF\rtarr \sF$ (which gives $\sF$ a structure of permutative category) and the obvious concatenation of $\otimes$-products of terms $\sC(j)$. By inspection, this product is associative and unital with unit the $1$-point component indexed on 
 $\sF(\mathbf{0},\mathbf{0})$.  Its composition
\[  \widehat{\sC}(\bn,\bp)\otimes \widehat{\sC}(\bm, \bn) \rtarr \widehat{\sC}(\bm,\bp) \]
is defined by sending the component indexed on $(\ps,\ph) \in \sF(\bn,\bp)\times \sF(\bm,\bn)$ to
the component indexed on $\ps\com \ph \in \sF(\bm,\bp)$
by the following composite, where $m_j = |\ph^{-1}(j)|$ for $1\leq j\leq n$,  $n_k = |\ps^{-1}(k)|$ for 
$1\leq k\leq p$, and $m_{k,j} = \sum_{\ps(j) = k} m_j$.   
$$\xymatrix{ 
 \bigotimes_{1\leq k\leq p}\sC(n_k) \otimes \bigotimes_{1\leq j\leq n}\sC(m_j)\ar[d]^{shuffle} \\
      \bigotimes_{1\leq k\leq p}\big( \sC(n_k)\otimes \bigotimes_{\ps(j) = k} \sC(m_j) \big) \ar[d]^{\ga^p}\\
      \bigotimes_{1\leq k\leq p} \sC(m_{k,j})\ar[d]^{\otimes_k\si_k(\ps,\ph)}\\ 
     \bigotimes_{1\leq k\leq p} \sC(m_{k,j}).\\}
$$
Here $\si_k(\ps,\ph)$ is the permutation of $|(\ps\com\ph)^{-1}(k)|$ elements that converts the natural order of $(\ps\com\ph)^{-1}(k)$ as a subset of $\{1,\cdots,m\}$ to the order obtained by regarding it as the set $\coprod_{\ps(j)=k}\ph^{-1}(j)$ so ordered that all elements of 
$\ph^{-1}(j)$ precede all elements of $\ph^{-1}(j')$ if $j< j'$ and each set $\ph^{-1}(j)$ has its natural order as a subset of $\{1,\cdots,m\}$.  
\end{defn}

We now have two symmetric monoidal category structures on $\overline{\sC}$, namely the structure obtained via the isomorphism with the permutative envelope $\overline{\sC}$ and the structure given as a subcategory of $\widehat{\sC}$.  By combinatorial inspection, these structures agree, giving the following result.

\begin{prop}  The permutative envelope of an operad in $\sV$ is a sub symmetric monoidal category of its category of operators.
\end{prop}

From $\widehat{\sC}$,  we construct a monad $\widehat{\mathbb{C}}$ in the category $\PI[\sV_*]$ of covariant functors $\PI \rtarr \sV_*$, where $\sV_*$ is the category of based objects in our cartesian monoidal category $\sV$.  We set

\begin{equation*}
\widehat{\bC} \sX  = \widehat{\sC} \otimes_{\PI} \sX,
\end{equation*}
which is spelled out as
\begin{equation}
  \label{eq:C-hat}
  (\widehat{\bC} \sX)(\bn)  = \widehat{\sC} (-, \bn) \otimes_{\PI} \sX.
\end{equation}

Here $\PI\subset \sF$ is the subcategory of injections and projections; thus $\ph\colon \bm\rtarr \bn$ is in $\PI$ if and only if $| \ph^{-1}(j)|$ is $0$ or $1$ for 
$1\leq j\leq n$.  When $\sC$ is a based operad, using $* \rtarr \sC(0)$ and the unit $\eta: *
  \rtarr \sC(1)$, we obtain a functor $\PI \rtarr \widehat{\sC}$, which is an inclusion if $\sC$ is $\SI$-free. Via this functor,
 $\widehat{\sC} (-, \bn)$ in \autoref{eq:C-hat} is a contravariant functor on $\Pi$. 
 Note that $\PI_{>0} = \PI\cap \sF_{>0}$ would add in the condition $\ph^{-1}(0) = \{0\}$, thus excluding  projections.  Therefore
 $\PI_{>0} = \LA$. 
 
 \begin{lem}  The evident inclusion
 $$ \overline{\sC} \otimes_{\LA} \sX \rtarr \widehat{\bC} \sX$$
 is an isomorphism.
\end{lem}
\begin{proof} Let $x\in \widehat{\sC}(\bm,\bn)$ be in the component indexed on $\ph \in \sF(\bm,\bn) \setminus \sF_{>0} (\bm, \bn)$
and let $y\in \sX(\bm)$.   Suppose that $\ph(i) = 0$ for $k$ numbers of the $i$ such that $1\leq i\leq m$.  Then there are maps 
$\pi\colon \bm\rtarr \mathbf{m-k}$ in $\PI$ and $\ps\colon \mathbf{m-k} \rtarr \bn$ in $\sF_{>0}$ such that $\pi(i) = 0$ when  $\ph(i) = 0$ and $\ph = \ps\com \pi$.  Moreover, it is immediate from the definitions of $\overline{\sC}$ and $\widehat{\sC}$ that there is a unique $w$ in the component of $\overline{\sC}(\mathbf{m-k}, \bn)$ indexed on $\ps$ such that $x = \pi^*(w)$.  By the definition of the coequalizer defining $\widehat{\bC} \sX$, $(x, y)$ is identified with $(w,\pi_*(y))$, which is in the component of 
$\overline{\sC} \otimes_{\LA} \sX$ indexed on $\ps$. 
\end{proof}
Nevertheless, use of the larger category $\widehat{\sC}$ rather than the permutative envelope $\overline{\sC}$ is essential to homotopical work in spaces (or $G$-spaces) since 
$\LA$ does not contain the Segal maps $\bn\rtarr \mathbf{1}$ in $\PI$ that are central to homotopical analysis.

\bibliographystyle{alpha}

\bibliography{references}

\end{document}